\documentclass[11pt,reqno]{amsart}
\usepackage{amsmath,amsthm,verbatim,enumerate}
\usepackage{bbm}

\pdfoutput=1

\usepackage[utf8]{inputenc}\usepackage{mathtools}
\usepackage{amsfonts}
\usepackage{amssymb}
\usepackage{algpseudocode}
\usepackage{xcolor}
\usepackage{graphicx}
\usepackage{subcaption}
\usepackage{float}

\newcommand{\set}[1]{\left\{#1\right\}}
\newcommand{\rsp}[1]{\langle#1\rangle}
\newcommand{\pderiv}[2]{\frac{\partial #1}{\partial #2}}
\newcommand{\deriv}[2]{\frac{d #1}{d #2}}

\newcommand{\R}{\mathbb{R}}

\newcommand{\bec}{\boldsymbol}
\usepackage{listings}
\usepackage{array}
\usepackage[backend=biber,style=numeric,sorting=none]{biblatex}
\newtheorem{theorem}{Theorem}
\newtheorem{lemma}{Lemma}[section]
\newtheorem*{corollary}{Corollary}
\theoremstyle{definition}
\newtheorem*{definition}{Definition}
\DeclareMathOperator*{\argmax}{arg\,max}
\DeclareMathOperator*{\argmin}{arg\,min}

\usepackage{enumitem}
\newenvironment{enro}{\begin{enumerate}[label=(\roman*)]}{\end{enumerate}}
\usepackage{ bbold }
\usepackage{hyperref}
\hypersetup{colorlinks}
\author[R. Hess]{Rowan Hess}
\address{Rowan Hess. 
Cornell University, Ithaca, NY 14853.}
\email{rdh248@cornell.edu}
\thanks{LL was partially supported by a grant from Open Philanthropy.}

\author[L.\ Levine]{Lionel Levine}
\address{Lionel Levine.
Department of Mathematics, Cornell University, Ithaca, NY 14853.
}
\email{levine@math.cornell.edu}

\begin{document}
%

\title{How to quantify the coherence of a set of beliefs}

\date{\today}

\begin{abstract}
    Given conflicting probability estimates for a set of events, how can we quantify how much they conflict? How can we find a single probability distribution that best encapsulates the given estimates? One approach is to minimize a loss function such as binary KL-divergence that quantifies the dissimilarity between the given estimates and the candidate probability distribution. Given a set of events, we characterize the facets of the polytope of coherent probability estimates about those events. We explore two applications of these ideas: eliciting the beliefs of large language models, and merging expert forecasts into a single coherent forecast.
\end{abstract}

\maketitle

\section{Introduction}
\noindent When two or more predictions conflict, what is a principled way to combine them into a single meta-prediction?
\\\\
This kind of question comes up in several different domains:

\begin{itemize}

\item Multiple experts are independently asked to make predictions, and we would like to integrate their predictions into a single meta-prediction that expresses the ``wisdom of the crowd'' \cite{galton_vox_1907, Ungar2012TheGJ, sempere2023alignment, mcandrew_aggregating_2021}.

\item A single individual may hold incoherent beliefs without being aware of the incoherence \cite{herzog_harnessing_2014}. However, when someone points out that their beliefs are incoherent they may want to update to the ``closest'' coherent set of beliefs. 

\item A market maker that has beliefs about underlying probabilities of events may want to give a coherent set of quotes so as to avoid giving a ``Dutch book" that is exploitable via risk-free arbitrage \cite{dutch_book}.
\item The expressed opinions of a large language model can be very sensitive to small changes in the prompt. To investigate whether the model holds an internally coherent belief and extract that belief, one approach is to ensemble many weak predictors derived from the model's internal state into a single meta-prediction \cite{burns2022discovering}.


\end{itemize}
    
\noindent Given a ground set $\Omega$, a list of events $\mathtt E = E_1, \ldots, E_n \subset \Omega$, and a list of \textit{credences} $q_1,\ldots,q_n$, the vector $\bec q = (q_1,\ldots,q_n)$ is called \textit{coherent} with respect to $\mathtt E$ if there is a probability measure $P$ (on the Boolean algebra generated by $E_1,\ldots,E_n$) such that $q_i = P(E_i)$ for all $i$.
\\\\
For example, when predicting the next day's weather, suppose three of our events are ``rainy'', ``warm and rainy'', ``cold and rainy''. Then we have the coherence condition  
    \[ P(\text{warm and rainy}) + P(\text{cold and rainy}) = P(\text{rainy}). \]
If ``warm or rainy'' is also one of our given events $E_i$, then we have the additional coherence condition
    \[ P(\text{warm}) + P(\text{rainy}) - P(\text{warm and rainy}) = P(\text{warm or rainy}). \]
These two examples of coherence conditions are both the result of probabilities being overdetermined: fixing the probabilities of some events determines the probabilities of others. There is also a second kind of coherence condition that results from probabilities being between by 0 and 1. For example, even if ``warm or rainy'' is not one of our events, we still have the coherence condition
    \[ P(\text{warm}) + P(\text{rainy}) - P(\text{warm and rainy}) \leq 1. \]

\noindent Given a list of events $\mathtt E$ and credences $\bec q$ we would like a systematic way of measuring incoherence: a loss function $L^*:[0, 1]^n\rightarrow \R\cup \set\infty$ satisfying $L^*(\bec q) \geq 0$, with equality if and only if $\bec q$ is coherent. One natural method of constructing such a loss is via some notion of ``distance'' from the vector $\bec q$ to the set $C(\mathtt E)$ of coherent vectors:
    \begin{equation} \label{eq:theloss} L^*(\bec q) = \inf \{L(\bec p,\bec q) \,:\, \bec p \in C(\mathtt E) \}. \end{equation}
We will discuss several choices of ``distance" function $L(\bec p,\bec q)$. We write distance in quotation marks because $L$ need not be a metric (e.g. it might not be symmetric in $\bec p$ and $\bec q$).
\\\\
The forecaster seeking to combine the beliefs of others (or individual seeking to correct their incoherent beliefs) can adopt the minimizer $\bec{p^*}$ of \eqref{eq:theloss}. In Theorem \ref{punique} below we give conditions for existence and uniqueness of $\bec{p^*}$.

\subsection{Plan of the paper}
We start Section \ref{sec:seperateBeliefs} defining the process of loss function minimization more precisely and giving conditions for when there is a unique minimizer $\bec{p^*}$. We proceed by summarizing a result from Predd et al.\ \cite{predd_probabilistic_2009} that can motivate one's choice of loss function and then by looking at two specific examples of loss functions (binary KL divergence and its transpose). We discuss natural justifications for these loss functions in Theorems \ref{thm:coherentBetter} and \ref{thm:fo}.
\\\\
Section \ref{sec:nn} explores a possible application of our method to eliciting beliefs from large language models. We use loss functions to generalize the method suggested in \cite{burns2022discovering} and consider possible modifications that this generalization allows.
\\\\
Section \ref{sec:inequalities} explores the geometry of $C(\mathtt{E})$, the set of coherent beliefs, which is a convex polytope.  It is easy to describe the vertices of this polytope (Lemma~\ref{lem:convex}), but somewhat more challenging to describe its facets (Theorem~\ref{thm:restatement}).
\\\\
In Section \ref{sec:indCohExp}, we consider how predictions should be aggregated in a setting where multiple experts each make internally coherent predictions. We seek a method to aggregate these predictions that leverages the internal coherence of each expert. 
\\\\
Finally, in Section \ref{sec:ex}, we illustrate these methods in the scenario of masked letter prediction and discuss the results.

\subsection{Related work}
This problem of the reconciliation has been examined in the past \cite{capotorti_correction_2010, thimm_inconsistency_2013}. In particular, \cite{capotorti_correction_2010} proposes the use of binary KL divergence as a loss function and discusses several of its properties.

\subsection{Preliminaries}
We will assume that there are $n$ events $E_i$ for $i\in \set{1, \dots, n}$ and that for each event $E_i$, we receive some estimate $q_i$ of its probability. These estimates may come from a single expert or multiple experts. We discuss the implications of one expert giving multiple beliefs in Section \ref{sec:indCohExp}.
\begin{definition}
A \textit{credence base} is the ordered pair $(\mathtt E, \textbf q)$ where $\mathtt E = (E_i)_{i=1}^{n}$ is a sequence containing the (potentially repeated) events whose probabilities are being estimated and $\bec q\in [0, 1]^n$ is the vector whose $i$th entry is $q_i$.
\end{definition}
\noindent
Recalling that every finite Boolean algebra is the power set of a set, let $2^\Omega$ be the Boolean algebra generated by $\mathtt E$, where $\Omega=\set{\omega_1, \dots, \omega_N}$ is finite as $\mathtt E$ is finite. For example, if $\mathtt E = (\set{\text{rock}}, \set{\text{paper}}, \set{\text{scissors},\text{rock}})$ then $\Omega = \set{\text{rock}, \text{paper}, \text{scissors}}$.
\begin{definition}
A credence base $(\mathtt E, \bec q)$ is \textit{coherent} if there exists a probability distribution $\pi:2^\Omega\rightarrow [0, 1]$ where $\pi(E_i)=q_i$ for all $i\in\set{1,\dots, n}$. We also call $\bec q$ \textit{coherent} with respect to $\mathtt E$ if $(\mathtt E, \bec q)$ is coherent.
\end{definition}
\begin{definition}
The \textit{set of coherent beliefs} over events $\mathtt E$ is         \[ C(\mathtt E)=\set{\boldsymbol p\in [0, 1]^n \,:\, (\mathtt E, \boldsymbol{p}) \text{ is coherent}}. \]
\end{definition}
\noindent 
\section{Loss Functions that Quantify Incoherence} \label{sec:seperateBeliefs}
\noindent In this section, we examine a class of loss functions $L(\bec p,\bec q)$ that are additive over events. We give conditions that guarantee that for any credence base $(\mathtt E,\bec q)$ there is a unique coherent credence base $(\mathtt E,\bec{p^*})$ minimizing the loss $L(\bec p,\bec q)$. Then we state a theorem of Predd et al.\ \cite{predd_probabilistic_2009} showing that in a certain forecasting setting it is never advantageous to submit an incoherent forecast. We proceed by giving natural justifications for 
two examples of loss functions: binary KL divergence and its transpose. 
\\\\
Recall that a function $f$ is \textit{lower semi-continuous} at $x$ if
    $$\liminf_{x_n\rightarrow x} f(x_n) \ge f(x).$$
Also, we say that a function $f$ is \textit{strictly convex when finite} if for all $t\in(0, 1)$ and $x, x'$ in its domain, then
    $$tf(x) + (1-t)f(x') \ge f(tx+(1-t)x')$$
    and the inequality is strict unless both sides are infinite.
\begin{definition}
    A \textit{dissimilarity function} $\ell(p, q)$ is a function $\ell:[0, 1]\times [0, 1]\rightarrow [0, \infty]$ satisfying 
    \begin{enro}
        \item $\ell$ is lower semi-continuous with respect to $p$
        \item $\ell$ is strictly convex when finite with respect to $p$
        \item $\ell(p, q)=0$ if and only if $p=q$.
    \end{enro} 
\end{definition}
Some examples of dissimilarity functions are:
\begin{itemize}
    \item $f(p, q) = p\ln \frac pq + (1-p) \ln \frac{1-p}{1-q}$
    \item $f^o(p, q) = f(q, p)$
    \item $\ell(p, q) = (p-q)^2$
    \item $\ell(p, q) = \left\{\begin{array}{cc}
        0 &  p=q \\
        \infty & p\ne q
    \end{array}\right..$
\end{itemize}
\noindent Given dissimilarity functions $\ell_1, \dots, \ell_n$, we define an associated \textit{loss function}
$L:[0, 1]^n\times [0, 1]^n\rightarrow [0, \infty]$ given by
\begin{equation}
    L(\textbf p, \textbf q)=\sum_{i=1}^n \ell_i(p_i, q_i).\label{eqn:lossfunction}
\end{equation}
\noindent We typically consider all dissimilarity functions $\ell_i$ to be equal to some common $\ell$. We denote the corresponding loss function $L_\ell$ and call it the loss function derived from dissimilarity function $\ell$. We omit the subscript when it is clear from context.
\\\\
Then, given a loss function $L$ with $n$ terms and a credence base $\mathcal{Q}=(\mathtt E, \bec q)$ with $n$ events, we can define the \emph{incoherence} of $\mathcal Q$ to be
\begin{equation}
    L^*(\boldsymbol q):=\min_{\boldsymbol p\in C(\mathtt E)} L(\textbf p, \textbf q).\label{eqn:lstar}
\end{equation}
We denote the minimizing coherent belief by $\bec p^*(\bec q)$. The next theorem, which extends \cite[Cor.~7]{capotorti_correction_2010} to a more general family of loss functions, gives conditions under which this minimizer is unique.

\begin{theorem}
\label{punique}
    Given a credence base $(\mathtt E, \bec q)$ and a loss function $L$ of the form in (\ref{eqn:lossfunction}), if there exists a coherent $\bec p\in C(\mathtt E)$ such that $L(\bec p, \bec q)$ is finite, then there exists a unique $\bec {p^*}$ such that
    $$L(\bec{p^*}, \bec q) = \min_{\bec p\in C(\mathtt E)} L(\bec p, \bec q).$$
    Moreover, $\bec{p^*} = \bec q$ if and only if $\bec q\in C(\mathtt E)$.
\end{theorem}
\begin{proof}
    As the sum of lower semi-continuous and strictly convex when finite functions, $L(\cdot, \bec q)$ is also lower semi-continuous and strictly convex when finite. Choose coherent $\bec p\in C(\mathtt E)$ such that $L(\bec p, \bec q)$ is finite. As the set $C(\mathtt E)$ is compact, these is a sequence $\bec p_1, \bec p_2, \dots C(\mathtt E)$ that converges to some $\bec p^*$ so that
    $$\lim_{k\rightarrow \infty} L(\bec p_k, \bec q) = \inf_{\bec p\in C(\mathtt E)} L(\bec p, \bec q).$$
    Then, by lower semi-continuity,
    $$\inf_{\bec p\in C(\mathtt E)} L(\bec p, \bec q) \le L(\bec p^*, \bec q) \le \lim_{k\rightarrow \infty} L(\bec p_k, \bec q) =\inf_{\bec p\in C(\mathtt E)} L(\bec p, \bec q) $$
    so $\bec p^*$ is a minimizer of $L(\cdot, \bec q)$ in $C(\mathtt E)$.
    \\\\
    As for uniqueness, if $L(\bec p, \bec q)=L(\bec p', \bec q)$ for $\bec p \ne \bec p'$,
    $$L\left(\frac 12 \bec p+\frac 12 \bec p', \bec q\right)<\frac 12 L(\bec p, \bec q)+\frac 12 L(\bec p', \bec q)=L(\bec p, \bec q)$$
    because $L(\bec p, \bec q)$ and $L(\bec p', \bec q)$ must be finite. Therefore, $\bec p$ is not a minimizer of $L(\cdot, \bec q)$.
    \\\\
    If $\bec q\in C(\mathtt E)$, then $L(\bec q, \bec q)=0$ by property (iii). Therefore, as $L$ can only take non-negative values and the minimizer is unique, $\bec{p^*} = \bec q$.
\end{proof}
\noindent We discuss how to compute $\bec p^*$ in Appendix \ref{sec:findpstar} and the continuity of this $\bec p^*$ and $L^*$ in Appendix \ref{sec:continuity}.
\subsection{The coherence theorem of Predd et al.}\label{section:predd}

The number $L^*(\bec q)$ is a way of quantifying the incoherence of the credence base $(\mathtt E, \bec q)$. But which dissimilarity function $\ell$ should we use?
One way to motivate the choice of $\ell$ arises from forecasting competitions: A forecaster submits a list of predicted probabilities $\bec q$ for events $\mathtt E$, and these are scored after the outcomes of all events are known. If event $E_i$ occurred, then the forecaster receives penalty $s(1, q_i)$ for that event; if $E_i$ did not occur, then the forecaster receives penalty $s(0, q_i)$ for that event.  The forecaster aims to minimize the expectation of the random variable
    \[S(\mathtt E, \bec q)=\sum_{i=1}^n s(1_{E_i}, q_i).\]

\noindent How should $s$ be chosen in order to elicit the true beliefs of each forecaster? If a forecaster believes that event $E_i$ will happen with probability $q_i$, then the scoring rule should incentivize them to submit prediction $q_i$ in order to minimize their expected score. A function $s$ with this property is called a \textit{proper scoring rule}. 
More formally, a proper scoring rule is a function $s:\set{0, 1}\times [0, 1] \rightarrow [0, \infty]$ satisfying \cite{predd_probabilistic_2009, winkler_good_1968}:
    \begin{enro}
        \item For each $p\in[0, 1]$, the quantity \[ ps(1, q)+(1-p)s(0, q)\] is uniquely minimized at $q=p$.
        \item $s$ is continuous in $q$. Note that we allow $s$ to take the value $+\infty$, but $(i)$ ensures that all values of $s$ are finite except possibly $s(0, 1)$ and $s(1, 0)$.
    \end{enro}
    
\noindent Predd et al.\ \cite{predd_probabilistic_2009} prove that for any proper scoring rule $s$, there is a dissimilarity function $\ell$ so that for any incoherent $\bec q$, the prediction $\bec{p^*}(\bec q)$ using $\ell$ scores better than $\bec q$ no matter which events actually happen! 
    Namely, let 
    \begin{equation}\ell(p, q)=p[s(1, q)-s(1, p)]+(1-p)[s(0, q)-s(0, p)]\label{eqn:ell}\end{equation}
    be the expected excess penalty for predicting $q$ rather than $p$ when the true probability of an event is $p$. Note that for loss functions $L^*$ derived from such dissimilarity functions, we have $L^*(\bec q)\ge 0$ with equality if and only if $(\mathtt E, \bec q)$ is coherent, because such functions $\ell$ are uniquely minimized when $p=q$ as $s$ is a proper scoring rule.
\begin{theorem}[\cite{predd_probabilistic_2009}]
   Let $s$ be any proper scoring rule, and let $\ell$ be given by (\ref{eqn:ell}).
   For any credence base $(\mathtt E, \bec q)$, let $\bec{p^*}(\bec q)\in C(\mathtt E)$ be the coherent minimizer of Theorem~\ref{punique}. 
   Then, for all $\omega \in \Omega$, $${S(\mathtt E, \bec q)(\omega)-S(\mathtt{E}, \bec{p^*}(\bec q))(\omega)\ge L^*(\bec q)}.$$
   \label{thm:coherentBetter}
\end{theorem}
\noindent Equation (\ref{eqn:ell}) still leaves a lot of possible choices of dissimilarity function $\ell$, since there are many choices of proper scoring rule $s$. Next we examine the case of the widely-used logarithmic scoring rule. The corresponding dissimilarity function $\ell$ turns out to be a variant of the Kulback-Leibler (KL) divergence \cite{kullback_information_1951}.

\subsection{Binary KL divergence}
Capotorti, Regoli, and Vattari \cite{capotorti_correction_2010} proposed correcting incoherent probability assessments using the binary KL-divergence
\begin{equation}
    f(p, q)=p \ln \frac pq + (1-p)\ln \frac{1-p}{1-q}.\label{eqn:kldivergence}
\end{equation}
The resulting loss $L(\bec p,\bec q) = \sum_{i=1}^n f(p_i,q_i)$ can be interpreted as the \emph{total expected excess surprise of all experts}, when expert $i$ predicts probability $q_i$ for event $E_i$ and the ground truth probability of $E_i$ is $p_i$. 
\\\\This loss function can also be motivated by the logarithmic scoring rule
$$s(i, q)=i\ln\frac 1 q + (1-i)\ln\frac{1}{1-q}$$
with the goal to minimize total score $S(\mathtt E, \bec q)=\sum_{i=1}^n s(1_{E_i}, q_i)$. By Theorem \ref{thm:coherentBetter}, if $(\mathtt E, \bec{q})$ is incoherent then the prediction $\bec{p^*}(\bec q)$ using $\ell=f$ scores strictly better than the prediction $\bec q$ no matter which events actually occur.

\subsection{Transposed binary KL divergence}
Next we consider the same loss with the roles of $\bec p$ and $\bec q$ interchanged:
$$f^{o}(p, q)=f(q, p).$$
In Table~\ref{fig:Lminimizer} we describe the coherent minimizer $\bec p^*$ determined by $f$ and $f^o$ in each of two basic scenarios:
    \begin{enumerate}
    \item Partition: Fix $n$ and let $\Omega = \set{\omega_1, \omega_2, \dots, \omega_n}$. For all $1\le i \le n$, we ;et $E_i = \set{\omega_i}$. In this scenario all experts give estimates for different events, exactly one of which occurs. 
    \item Repetition: Let $\Omega = \set{\omega_1,\omega_2}$ and for all $1\le i \le n$, let $ E_i=\set{\omega_1}$. In this scenario all experts give estimates for the same event.
    \end{enumerate}
\begin{table}[h]
    \begin{center}
    \begin{tabular}{|m{4.5em} | m{16.5em}| m{13em} |}
         \hline
         \cline{2-3}& Binary KL loss & Transposed binary KL loss \\
         \hline
         Partition
         & $\ln\frac{p^*_i}{1-p^*_i}-\ln\frac{q_i}{1-q_i}$ does not depend on $i$&  $\frac{q_i-p^*_i}{p^*_i(1-p^*_i)}$ does not depend on $i$\\ 
         \hline
         Repitition
         & $\frac{p^*}{1-p^*}=\left(\prod_{i} \frac{q_i}{1-q_i}\right)^{1/n}$& $p^*=\frac{1}{n}\sum_{i} q_i$
         \\\hline
         \end{tabular}
         \caption{Description of the coherent minimizer $\bec{p^*}$ in each of two scenarios, with two different loss functions: Binary KL loss ($\ell=f$ as in (\ref{eqn:kldivergence})) and its transpose $(\ell=f^{o})$.}
         \label{fig:Lminimizer}
\end{center}
\end{table}
\noindent The motivation for switching the order of the variables in $f$ comes from the following theorem, showing that the minimizer of loss functions determined by $f^{o}$ is now, in a certain scenario, the maximum likelihood estimation of $\boldsymbol{p}$.
\\\\
Suppose the $i$th expert independently observes some number $w_i$ of independent trials of event $E_i$, and submits the estimate $q_i = \frac{k_i}{w_i}$ where $k_i$ is the number of trials in which $E_i$ occurred. Note that if $P(E_i)=p_i$, then $k_i$ has the binomial distribution $\text{Bin}(w_i,p_i)$.
\\\\
To represent that each expert has a different amount of information, we can multiply each summand in (\ref{eqn:lossfunction}) by some weight $w_i$.


\begin{theorem}\label{thm:fo}
The maximum likelihood estimator $\bec{\hat p}$ given $\boldsymbol q$ is 
    \[\argmin_{\boldsymbol{p}\in C(\mathtt E)}\sum_{i=1}^nw_if^{o}(p_i, q_i).\] 
\end{theorem}

\begin{proof}
    The probability of receiving the estimations $\boldsymbol q$ is
    $$P(\boldsymbol q | \boldsymbol p)=\prod_{i=1}^n\binom{w_i}{q_iw_i}p_i^{q_iw_i}(1-p_i)^{(1-q_i)w_i}.$$
    Next, if $\boldsymbol{\hat p}$ is the maximum likelihood estimation, then
    $$\boldsymbol{\hat p}= \argmax_{\boldsymbol p\in C(\mathtt E)}\prod_{i=1}^n\binom{w_i}{q_iw_i}p_i^{q_iw_i}(1-p_i)^{(1-q_i)w_i}$$
    $$=\argmax_{\boldsymbol p\in C(\mathtt E)}\sum_{i=1}^nq_iw_i\ln p_i+(1-q_i)w_i\ln (1-p_i)$$
    $$=\argmax_{\boldsymbol p\in C(\mathtt E)}\sum_{i=1}^nw_i(q_i\ln p_i + (1-q_i)\ln(1-p_i))$$
    $$=\argmax_{\boldsymbol p\in C(\mathtt E)}\sum_{i=1}^nw_i\left(q_i\ln \frac{p_i}{q_i} + (1-q_i)\ln \frac{1-p_i}{1-q_i}\right)$$
    $$=\argmin_{\boldsymbol p\in C(\mathtt E)}\sum_{i=1}^nw_i\left(q_i\ln \frac{q_i}{p_i} + (1-q_i)\ln \frac{1-q_i}{1-p_i}\right)$$
    $$=\argmin_{\boldsymbol p\in C(\mathtt E)}\sum_{i=1}^nw_if^o(p_i, q_i).\qedhere$$
\end{proof}
\noindent In general, $f$ produces more extreme probabilities than $f^o$ and gives more weight to extreme probability estimates, as demonstrated numerically in Appendices \ref{sec:compare} and \ref{sec:ex}.
\section{Application: Eliciting Latent Beliefs from Language Models} \label{sec:nn}
\noindent Large language models sometimes express beliefs that they ``know'' to be false, for example because the prompt includes false statements. Burns et al \cite{burns2022discovering} propose a method to elicit the model's actual belief about a natural language claim corresponding to an event $E$ from its hidden state $\phi(\text{des}_E)$ where $\text{des}_E$ is an assertion in natural language that the event $E$ holds. This method learns a linear probe $b(\phi)$ minimizing the loss
    \begin{equation}
        (q_{E} + q_{E^c} - 1)^2 + \min(q_{E},q_{E^c})^2\label{eqn:burns}
    \end{equation} 
where $q_{E} =\sigma(b(\phi(\text{des}_E)))= \frac{1}{1+e^{-b(\phi(\text{des}_E))}}$ and $q_{E^c} = \sigma(b(\phi(\text{des}_{E^c})))$ are intended to estimate the model's credence in the claims $\text{des}_E$ and $\text{des}_{E^c}$. The idea is that the term $(q_{E} + q_{E^c} - 1)^2$ is a measure of the incoherence of the linear probe, while the $\min(q_{E},q_{E^c})^2$ term encourages the linear probe to be decisive. We propose three possible elaborations of this approach, all of which train probes to minimize an expression of the form
$$\mathcal I(\bec q) + \mathcal{J}(\bec q)$$
where $\mathcal I$ is a measure of incoherence and $\mathcal J$ is a measure of indecisiveness.
\subsection{The Case of a Single Event and its Complement}\label{subsec:singleEvent}
The following proposition shows that the measure of incoherence $(q_E+q_{E^c}-1)^2$ can be derived from the dissimilarity function $\ell(p, q)=2(p-q)^2$.
\begin{lemma}
    When using dissimilarity function $\ell(p, q)=2(p-q)^2$, the incoherence function becomes
    $$L_\ell^*(q_E, q_{E^c}) = (q_{E} + q_{E^c} - 1)^2.$$
\end{lemma}
\begin{proof}
    See Lemma \ref{lem:burns_loss} in the appendix.
\end{proof}
\noindent This approach could be generalized to choosing $\ell$ to be either the functions $f$ or $f^o$ discussed previously. The next two lemmas give approximations for $L^*$ when $q_E+q_{E^c}-1$ is small
\begin{lemma}
    When using dissimilarity function 
    $$\ell(p, q)=f(p, q)=p\ln \frac pq + (1-p)\ln \frac{1-p}{1-q}$$
    the loss function becomes
    $$L_f^*(q_E, q_{E^c}) = \frac{(q_E + q_{E^c} - 1)^2}{(1-q_E+q_{E^c})(q_E+1-q_{E^c})}+O((q_E + q_{E^c} - 1)^3).$$\label{lem:burns_f2}
\end{lemma}
\begin{proof}
    See Lemma \ref{lem:burns_f} in the appendix.
\end{proof}
\begin{lemma}
    When using dissimilarity function 
    $$\ell(p, q)=f^o(p, q)=q\ln \frac qp + (1-q)\ln \frac{1-q}{1-p}$$
    the loss function becomes
    $$L_{f^o}^*(q_E, q_{E^c}) = \frac{(q_E + q_{E^c} - 1)^2}{(1-q_E+q_{E^c})(q_E+1-q_{E^c})}+O((q_E + q_{E^c} - 1)^3).$$\label{lem:burns_fo2}
\end{lemma}
\begin{proof}
    See Lemma \ref{lem:burns_fo} in the appendix.
\end{proof}
\noindent The reason that the $(1-q_E+q_{E^c})(q_E+1-q_{E^c})$ term might be desirable to have in the denominator is that it increases the loss for the same absolute error when the probabilities are more extreme, which mirrors the intuition behind common scoring functions such as log-loss.
\subsection{Multiple Probes and Multiple Events}
A natural extension of \cite{burns2022discovering} is to train $k$ probes $b_1, \dots, b_k$ where $b_i$ is trained to elicit the model's credence from its $i$th layer hidden state $\phi_i$ using $2m$ rephrasings $\text{des}_{E, j}$ and $\text{des}_{E^c, j}$ of a natural language assertion about an event $E$ and its complement $E^c$. Then, letting $q_i(\text{des}) = \sigma(b_i(\phi_i(\text{des})))$, one could replace the $(q_E + q_{E^c}-1)^2$ term of (\ref{eqn:burns}) with
$$\mathcal I(\bec q) = L^*\left(q_i(\text{des}_{E, j}), q_i(\text{des}_{E^c, j}) \right)_{i=1}^{k}\left._{j=1}^m\right..$$
This could be further generalized to a set of $n$ interrelated events $\set{E_1, \dots, E_n}$ each of which has $m$ natural language assertions $\text{des}_{E_h, 1},  \dots, \text{des}_{E_h, m}$ that it holds, using the coherence term
$$\mathcal I(\bec q) = L^*\left(q_i(\text{des}_{E_h, j})\right)_{i=1}^{k}\left._{h=1}^n\right.\left._{j=1}^{m}\right..$$
Theorem \ref{thm:holderImpliesDiff} in Appendix \ref{sec:continuity} gives a formula for the gradient of $L^*$ assuming $\bec p^*$ is sufficiently smooth.
\\\\
Why would multiple probes ($k> 1$) help elicit beliefs? We hypothesize that the language model's beliefs -- when they exist -- arise from a bundle of internal heuristics that sometimes conflict.  Each probe $b_i$ represents one such heuristic. In the case that $L^*$ is close to zero, the model's internal heuristics are mostly self-consistent (the credences are close to an actual probability distribution) and in this case we might say that the model ``has beliefs" about the claims $E_1,...,E_n$. These beliefs could be certain or uncertain: if all the probes think a coin flip is 50\% likely to land heads, then that is a coherent belief. On the other hand, if $L^*$ is large, then it is possible that the model ``does not have beliefs" about the claims $E_i$, or else that the probes simply did not discover the right heuristics.
\\\\
Why would multiple events ($n> 2$) help elicit beliefs? By minimizing $L^*$ applied to more events, we are encouraging not just $q_E+q_{E^c}=1$ but also $q_{E\cup F} = q_E + q_F - q_{E\cap F}$ and so on. The philosophy is that to find ``beliefs", one should look for functions of the hidden state that obey the laws of probability. The logical dependencies between events are incorporated in the definition of $L^*$ as a minimum over $C(\mathtt E)$, as defined in equation (\ref{eqn:lstar}). We further study the geometry of $C(\mathtt E)$ in Section \ref{sec:inequalities}.
\subsection{Generalizing the Decisiveness Term}
Recall that the $\min(q_{E},q_{E^c})^2$ term in Equation (\ref{eqn:burns}) is present to encourage the linear probe to be decisive. However, in the scenarios above with more than 2 estimates, it is unclear what should replace it. We discuss some options below:
\begin{itemize}
    \item We want the probes to be decisive so as to maximize information they give about the language model's internal state. One way to quantify the information content of $\bec p^*$ is by the maximum entropy of a probability distribution that gives $\bec p^*$. Hence, one could use the decisiveness term
    $$\mathcal J(\bec q) = \max \set{H(\pi) \, : \pi(\bec{ \mathtt E}) = \bec p^*(\bec q)}$$
     where
    $$H(\pi) = \sum_{\omega\in \Omega} -\pi(\omega) \ln \pi(\omega).$$
    \item The above notion can be generalized. Recall that ${s(i, q)=i\ln q + (1-i) \ln (1-q)}$ is the log-loss scoring rule. Then, we can write entropy as
    $$H(\pi) = E_\pi\left[ \sum_{\omega\in \Omega}1_\omega s(1, \pi(\omega))\right].$$
    Recall that in Section \ref{section:predd}, we discussed how a proper scoring rule $s$ can be transformed into a loss function $\ell$. If using such a loss function, it might make sense to replace the log-loss function above with the proper scoring rule that $\ell$ is based on, making the decisiveness term
    $$\mathcal J(\bec q) = \max \set{E_\pi\left[ \sum_{\omega\in \Omega}1_\omega s(1, \pi(\omega))\right]: \pi(\bec{ \mathtt E}) = \bec p^*(\bec q)}$$
    for a general proper scoring rule $s$.
    \item Using the idea that the least decisive distribution $\pi$ is the distribution that maximizes 
    $$E_\pi\left[ \sum_{\omega\in \Omega}1_\omega s(1, \pi(\omega))\right],$$
    let $\bec u = \pi(\mathtt E)$ be this least decisive set of coherent beliefs and use either $\mathcal J(\bec q) = -L(\bec u, \bec p^*(\bec q))$ or $\mathcal J(\bec q) = -L(\bec p^*(\bec q), \bec u)$  to reward the probe for being more decisive.\\
    \item When using $f$ or $f^o$ as a measure of dissimilarity, the incoherence term by itself might be enough to incentivize decisiveness due to the increased sensitivity of the sigmoid function to noise when its input is close to $0$.
    To illustrate this, we revisit the scenario where we use one linear probe to find $q_E$ and $q_{E_c}$ from one natural language description of the event $E$ and one of its complement. We suppose that  that the value of the linear probe $b(\phi(\text{des}_E))$ has normal distribution $N(\sigma^{-1}(p), S^2)$ and $b(\phi(\text{des}_{E^c}))\sim N(\sigma^{-1}(1-p), S^2)$ independently where $p$ is the probe's goal credence in event $E$. Regardless of which of $f$ or $f^o$ we choose  as a dissimilarity function the expected value of the incoherence becomes approximately, using the quadratic approximation given by Lemmas \ref{lem:burns_f2} and \ref{lem:burns_fo2}:
    $$\mathbb E[L^*((q_{E}, q_{E^c}))]\approx \frac{\mathbb E[(q_{E} + q_{E^c} - 1)^2]}{4p(1-p)}.$$
    Noting $q_{E} + q_{E^c} - 1$ has mean $0$ and is the sum of two independent random variables with approximate variances $\sigma'(\sigma^{-1}(p))^2S^2=\sigma'(\sigma^{-1}(1-p))^2S^2$ by the $\delta$-method, this becomes
    $$\mathbb E[L^*((q_{E}, q_{E^c}))]\approx \frac{2 S^2}{4p(1-p)}\sigma'(\sigma^{-1}(p))^2=\frac{S^2p(1-p)}{2}.$$
    As this is minimized when $p$ is close to $0$ or $1$, the incoherence term by itself incentivizes the probe to be decisive.
\end{itemize}
\section{The Polytope of Coherent Beliefs} \label{sec:inequalities}
\noindent Here, we explore the shape of the space of coherent beliefs. We first describe it as the convex hull of a finite set of $0,1$-vectors in Lemma \ref{lem:convex} and end by describing the inequalities that bound this polytope in Theorem \ref{thm:restatement}.
\\\\
Given a set of events $\mathtt E$ with beliefs $\bec q$, let $\textbf E_i\in \set{0, 1}^N$ be the vector whose $j$th entry is $1$ if $\omega_j\in E_i$ and $0$ otherwise. Then, we can encode the structure of events $\mathtt E$ in a matrix $V$ whose $i$th row is the vector $\bec E_i$. Because $V$ contains the same information as $\mathtt E$, it will be convenient to use the pair $(V, \bec q)$ to represent the credence base $(\mathtt E, \bec q)$. We also encode a probability distribution $\pi$ on $\Omega$ with the vector $\bec \pi \in [0, 1]^N$ whose $j$th entry $\bec \pi \cdot \bec e_j = \pi(\set{\omega_j})$.
\begin{definition}
    A vector $\bec\pi\in [0, 1]^N$ is a probability vector if $\sum_{i=1}^N \pi_i=1$.
\end{definition}
\begin{lemma}
    \cite{capotorti_correction_2010} A credence base $\mathcal{Q}=(V, \bec q)$ is coherent if and only if $\boldsymbol q$ is in the convex hull of the columns of $V$. \label{lem:convex}
\end{lemma}
\begin{proof}
    By definition $\mathcal Q$ is coherent if and only if there is a probability vector $\bec\pi \in [0, 1]^{N}$ where $V \bec \pi = \boldsymbol{ q}$. If $V_i$ is the $i$th column of $V$, that is true if and only if
    $$\boldsymbol q = \sum_{i=1}^{N} \pi_i \boldsymbol V_i.\qedhere$$
\end{proof}
\noindent For example, if we receive three beliefs: one about the probability of it being warm, one about the probability of it raining, and one of the probability of it being warm and raining, then the matrix $V$ becomes
$$\begin{pmatrix}
    1 & 1 & 0 & 0\\
    1 & 0 & 1 & 0\\
    1 & 0 & 0 & 0
\end{pmatrix}.$$
So the polytope of coherent beliefs is the tetrahedron depicted in Figure \ref{fig:poytopeVisual}.
\begin{figure}[h]
    \begin{subfigure}[b]{.4\textwidth}
        \centering
        \includegraphics[width=\linewidth]{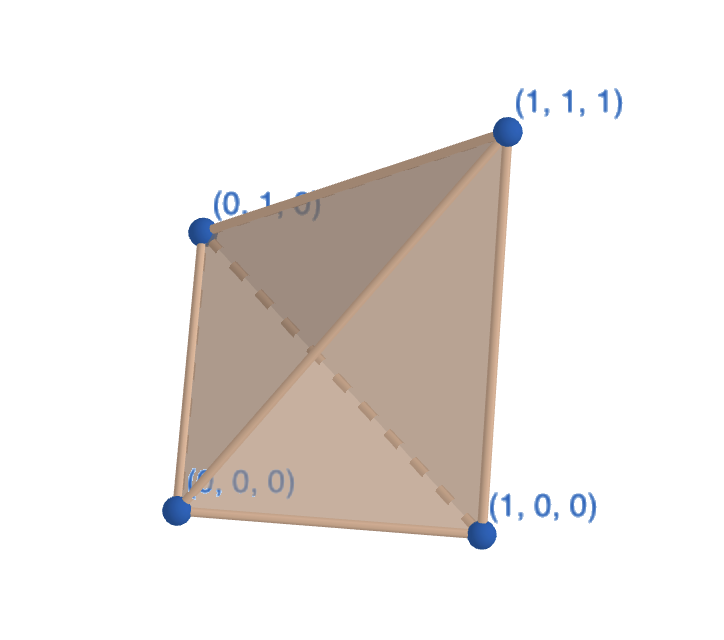}
    \end{subfigure}
    \begin{subfigure}[b]{.5\textwidth}
\begin{itemize}[leftmargin=*]
    \item The facet not containing vertex $(1, 1, 1)$ corresponds to the inequality $P(\text{warm and rainy})\ge 0$.
    \item The facets not containing either $(1, 0, 0)$ or $(0, 1, 0)$ correspond to the inequalities $P(\text{warm})\ge P(\text{warm and rainy})$ and $P(\text{rainy})\ge P(\text{warm and rainy})$.
    \item The facet not containing vertex $(0, 0, 0)$ corresponds to the inequality ${1 + P(\text{warm and rainy}) \ge P(\text{warm}) + P(\text{rainy})}$.
\end{itemize}
    \end{subfigure}
    \caption{The polytope of coherent beliefs about $P(\text{warm}), P(\text{rainy})$, and $P(\text{warm and rainy})$ and the inequalities corresponding to its facets as described in Theorem \ref{thm:restatement}.}
    \label{fig:poytopeVisual}
\end{figure}
\noindent \\Let $\bar V$ be the $(n+1)\times N$ matrix obtained by appending a final row of $1$s to $V$. Similarly, take $\bar{\textbf q}\in [0, 1]^{n+1}$ to have the same first $n$ entries as $\textbf q$ and $n+1$th entry $1$ and $E_{n+1}=\bec 1$, the vector whose entries are all $1$s. These represent the fact that $\pi(\Omega)=1$. For two vectors $\bec v, \bec w$, write $\textbf u \geq \bec v$ if $u_i\ge v_i$ for all $i$. Finally, let $O_+$ be $\R_{\ge 0}^N$ and for a matrix $M$, let $\rsp{M}$ denote the row span of $M$. 
\\\\
The following is a statement of the Dutch Book Theorem \cite{dutch_book}.
\begin{lemma}
    The following are equivalent:
    \begin{enro}
        \item $\mathcal{Q}$ is coherent.
        \item For all $(a_1, ..., a_{n+1})\in \R^{n+1}$, if $\sum_{i=1}^{n+1}a_i\boldsymbol E_i\geq \boldsymbol 0$, then $\sum_{i=1}^{n+1}a_i \bar q_i\ge 0$.
    \end{enro} \label{lem:qcoherent}
\end{lemma}
\begin{proof}
    Suppose $(ii)$. By Farkas' Lemma \cite{Farkas}, as there is no $\bec a \in \R^{n+1}$ with all entries in $\bar V^T \bec a$ non-negative and $\bec a \cdot \bec{ \bar q} < 0$, it must be that there is a solution $\bec \pi \in O_+$ to $\bar V \bec \pi = \bec{\bar q}$, with the row of $1$s in $\bar V$ ensuring that $\bec \pi$ is a probability vector.
    \\\\
    For the converse, consider $\mathcal{Q}$ coherent. Then, there exists $\boldsymbol \pi \in [0, 1]^N$ where $\bec E_i\cdot \boldsymbol \pi = \bar q_i$ for all $1\le i \le n+1$. But, if $\bar V^T \boldsymbol a \ge \bec 0$, then
    $$\boldsymbol a \cdot \boldsymbol{\bar q}=\boldsymbol a \cdot (\bar V \boldsymbol \pi)=(\bar V^T \boldsymbol a) \cdot \boldsymbol \pi\ge0$$
    as the dot product of two non-negative vectors.
\end{proof}
\begin{lemma}
    Suppose that $\bar{V}^T\boldsymbol a=\boldsymbol 0$ for some non-zero vector $\boldsymbol a \in \R^{n+1}$ where $\boldsymbol{\bar q} \cdot \boldsymbol a \neq 0$. Then, $\mathcal{Q}$ is incoherent.
    \label{lem:ConsistencyOnRemoval1}
\end{lemma}
\begin{proof}
    If $\boldsymbol{\bar q} \cdot \boldsymbol a<0$, then we have $\bar V^T\boldsymbol a =\bec 0 \ge \bec 0$ but $\bec{\bar q} \cdot \boldsymbol a<0$. By Lemma \ref{lem:qcoherent}, $\mathcal{Q}$ is incoherent.
    \\\\
    If $\boldsymbol {\bar q} \cdot \boldsymbol a>0$, then $\bec{\bar q}\cdot -\bec a<0$, so $\mathcal{Q}$ is incoherent.
\end{proof}
\begin{lemma}
    Suppose that $\bar{V}^T\boldsymbol a=\boldsymbol 0$ for some non-zero vector $\boldsymbol a \in \R^{n+1}$ where $\boldsymbol{\bar q} \cdot \boldsymbol a = 0$. For any $1\le i \le n$ where $a_i\ne 0$, we can remove the $i$th entry of $\boldsymbol{ q}$ and the $i$th row from $ V$ to obtain $\boldsymbol{ r}$ and $ W$ so that the credence base $(\boldsymbol r, W)$ is coherent if and only if $\mathcal{Q}$ is coherent.
    \label{lem:ConsistencyOnRemoval2}
\end{lemma}
\begin{proof}
    Without loss of generality, $\boldsymbol E_1 = \sum_{i=2}^{n+1} c_i \boldsymbol E_i$ and $q_1=\sum_{i=2}^{n+1} c_i q_i$. Call this vector $\boldsymbol c\in R^n$. Consider removing the first entry of $\boldsymbol q$ and the first row of $ V$ to yield $\boldsymbol r$ and $W$.
    \\\\
    For any $\boldsymbol a$ where $\bar V^T \boldsymbol a \ge \bec 0$, if $\boldsymbol a_p$ is is the vector $\boldsymbol a$ with the first entry removed, then $\bar V^T \boldsymbol a = \bar W^T(\boldsymbol a_p + \boldsymbol c)$. Then, also $\boldsymbol{\bar q} \cdot \boldsymbol a = \boldsymbol{\bar r} \cdot (\boldsymbol a_p+\boldsymbol c)$. Therefore, if $(\boldsymbol r, W)$ is coherent, then $\mathcal{Q}$ is coherent.
    \\\\
    If $(\boldsymbol r, W)$ were incoherent, then $\sum_{i=2}^{n+1}c_i \boldsymbol E_i\geq 0$ but $\sum_{i=2}^{n+1}c_i q_i< 0$, so $\mathcal{Q}$ is also incoherent.
\end{proof}
\noindent
Based on Lemmas \ref{lem:ConsistencyOnRemoval1} and \ref{lem:ConsistencyOnRemoval2}, from here on, we will only consider  $\bar V$ of rank $n+1$. Then, as the rows of $\bar V$ are linearly independent, the function $Q:\rsp{\bar V}\rightarrow \R$ where $Q(\boldsymbol o)=\boldsymbol a \cdot \boldsymbol{\bar q}$ if $\bar V ^T \boldsymbol a = \boldsymbol o$ is well-defined. By Lemma \ref{lem:qcoherent}, $\mathcal{Q}$ is coherent if and only if $Q(\boldsymbol o)\geq 0$ for all $\boldsymbol o \in \rsp{\bar V}\cap O_+$.
\\\\
Vectors $\bec a\in \R^{n+1}$ can be thought of as ``bets", where the payout of the bet in world $\omega$ is equal to
$$\sum_{i=1}^{n+1} a_i 1_{E_i}(\omega)$$
where $E_{n+1}=\Omega$. We consider the bets to be against a bookmaker with credence base $\mathcal Q$ who expects no edge.
\\\\
For such a bet $\bec a$, the vector $\bec b = \bar V^T \bec a\in \R^N$ represents the payout of $\bec a$ atomwise: if atom $\omega_j$ happens, then the payout of bet $\bec a$ will be $b_j=\bec e_j \cdot (\bar V^T \bec a)$. Because we assume that $\bar V$ has full rank, there is a bijection between bets $\bec a\in \R^n$ and atomwise payouts $\bec b\in \rsp{\bar V}$. Note that it does not make sense to talk about atomwise payouts outside of $\rsp{\bar V}$ as it is impossible to make bets with such payouts.
\\\\
For finding the coherence of a credence base, it is somewhat more helpful to think in terms of atomwise payouts rather than bets. For instance, if all entries of an atomwise payout $\bec b$ are non-negative, then, if $\mathcal Q$ is coherent, the cost $Q(\bec b)$ of the bet should be non-negative. The Dutch Book Theorem \ref{lem:qcoherent} tells us that having this condition for all atomwise payouts $\bec b\in \rsp{\bar V}$ is equivalent to coherence. Recalling that the set of coherent beliefs is a polytope, each bet whose atomwise payout is non-negative corresponds to an inequality that all elements of the polytope satisfy, and if a belief is outside the polytope, it violates an inequality corresponding to some facet, which is a bet that can be made. The following lemma allows us to just consider a positive spanning set $B$ of atomwise payouts that need to be checked to ensure coherence.
\begin{definition}
    A \textit{positive spanning set} of a set $S\subseteq \R^N$ is a subset $B\subseteq S$ with the property that for any $\boldsymbol s \in S$, there exist coefficients $c_{\boldsymbol b}\in \R_{\ge 0}$ for each $\boldsymbol b \in B$ where only a finite number of $c_{\bec b}$ are non-zero and
    $$\boldsymbol s = \sum_{\boldsymbol b \in B} c_{\boldsymbol b}\boldsymbol b.$$
\end{definition}
\begin{lemma}
    Let $B$ be a positive spanning set of $\rsp{\bar V}\cap O_+$. Then, $\mathcal{Q}$ is coherent if and only if $Q(\boldsymbol b)\geq 0$ for all $\boldsymbol b \in B$.
\end{lemma}
\begin{proof}
    If $Q(\boldsymbol b)<0$, for some $\boldsymbol b \in B$, then $\mathcal{Q}$ is incoherent by Lemma \ref{lem:qcoherent}.
    \\\\
    Suppose that $Q(\boldsymbol b)\ge 0$ for all $\boldsymbol b \in B$. Then, consider any $\boldsymbol o\in \rsp{\bar V}\cap O_+$. As $B$ is a positive spanning set, we can represent
    $$\boldsymbol o =\sum_{\boldsymbol b \in \mathcal{B}} c_{\boldsymbol b}\boldsymbol b.$$
    As $Q$ is linear, 
    $$Q(\boldsymbol o)=\sum_{\boldsymbol b \in \mathcal{B}} c_{\boldsymbol b}Q(\boldsymbol b)\ge 0$$
    as the sum of the product of non-negative reals. Therefore, $\mathcal{Q}$ is coherent.
\end{proof}
\begin{definition}
    The set of \textit{extremal vectors} of $\rsp{\bar V}\cap O_+$ is
    $$M:=\set{\boldsymbol b\in \rsp{\bar V}\cap O_+:\forall \boldsymbol o \in \rsp{\bar V}\cap O_+, \bec o\le \bec b \implies \boldsymbol o = \lambda \boldsymbol b \text{ for some $\lambda \in \R$}}.$$
\end{definition}
\noindent Let $O_1=\set{\boldsymbol o \in \rsp{\bar V}\cap O_+:\text{the first non-zero entry of }\boldsymbol o \text{ is 1}}$ where we fix an arbitrary ordering of the finite number of entries of $\bec o$. We will show that $M\cap O_1$ is a minimal positive spanning set of $\rsp{\bar V} \cap O_+$.
\\\\
\noindent If we think of each positive spanning set as a set of atomwise payouts that the bookmaker has to check is positive to ensure coherence, the following gives some atomwise payouts that the bookmaker definitely has to check.
\begin{lemma}
    Let $B$ be a positive spanning set of $\rsp{\bar V} \cap O_+$ where $B\subseteq O_1$. Then, $$O_1\cap M\subseteq B.$$
\end{lemma}
\begin{proof}
    If $B$ did not include such a $\boldsymbol b\in O_1\cap M$, as $B$ is a positive spanning set, then there is a finite set $S$ where $\boldsymbol b = \sum_{s\in S} c_s\boldsymbol b_s$ for some $c_s\geq 0$. But as not all $c_s = 0$, without loss of generality, suppose $c_1\neq 0$. Then, $\boldsymbol b - c_1\boldsymbol b_1\geq 0$, so $\boldsymbol b = \lambda \boldsymbol b_1$. As the first non-zero entry of each is $1$, $\boldsymbol b = \boldsymbol b_1$.
\end{proof}
\begin{definition}
    Let $A$ be a subset of $\R^n$. Consider $Z:A\rightarrow 2^{\set{1, \dots, n}}$ where $Z(\boldsymbol o)=\set{i:o_i=0}$. A vector $\boldsymbol v$ is \textit{maximally-$0$} in $A$ if there there is no $\boldsymbol w\in A\symbol{92}\set{\boldsymbol 0}$ where $Z(\boldsymbol v)\subsetneq Z(\boldsymbol w)$.
\end{definition}
\begin{lemma}
    For any vector $\boldsymbol o$ that is maximally-$0$ in $\rsp{\bar V}$, if $Z(\boldsymbol v)=Z(\boldsymbol o)$ for some vector $\boldsymbol v\in \rsp{\bar V}$, then $\boldsymbol v \in \rsp{\boldsymbol o}$. \label{lem:z=maximally0}
\end{lemma}
\begin{proof}
    Suppose that $\boldsymbol v \notin \rsp{\boldsymbol o}$. Let the first non-zero entry of $\boldsymbol o$ be $\lambda$ and the first non-zero entry of $\boldsymbol v$ be $\mu$. Then, $\boldsymbol 0\neq\boldsymbol o - \frac \lambda \mu \boldsymbol v$ and $Z(\boldsymbol o) \subsetneq Z(\boldsymbol o-\frac \lambda \mu\boldsymbol v)$.
\end{proof}
\noindent It turns out that it suffices to just check the maximally-0 elements of $\rsp{\bar V}$, as shown by the following to lemmas.
\begin{lemma}
    A vector $\boldsymbol o \in \rsp{\bar V}\cap O_+$ is maximally-$0$ in $\rsp{\bar V}$ if and only if $\boldsymbol o \in M$\label{lem:MequivMax0}.
\end{lemma}
\begin{proof}
    Suppose that $\boldsymbol o \not \in M$. Then, there exists  $\boldsymbol v \in \rsp{\bar V}\cap O_+\symbol{92}\rsp{\boldsymbol o}$ where $\boldsymbol o - \boldsymbol v\geq 0$. Then, $Z(\boldsymbol o) \subseteq Z(\boldsymbol v)$, so by Lemma \ref{lem:z=maximally0}, $\boldsymbol o$ is not maximally-0 in $\rsp{\bar V}$.
    \\\\
    Suppose that $\boldsymbol o\in M$. If $\boldsymbol o$ were not maximally-$0$ in $\rsp{\bar V}\cap O_+$, then we would have $\boldsymbol v\in \rsp{\bar V}\cap O_+\setminus\set{\bec 0}$ with with $Z(\bec o) \subsetneq Z(\bec v)$. Let $v_{\uparrow}$ be the maximum entry of $\boldsymbol v$ and $o_{\downarrow}$ be the minimum entry of $\boldsymbol o$. Then, $\boldsymbol o - \frac{o_{\downarrow}}{v_{\uparrow}}\boldsymbol v \geq \bec 0$, which is impossible as $\boldsymbol v$ cannot be in the span of $\boldsymbol o$ by the definition of $M$. Therefore, $\bec o$ is maximally-$0$ in $\rsp{\bar V}\cap O_+$ by contradiction.
    \\\\
    Next, it suffices to show that if $\bec o$ is maximally-$0$ in $\rsp{\bar V}\cap O_+$, then $\bec o$ is maximally-$0$ in $\rsp{\bar V}$. To see this, suppose not and that there exists some $\boldsymbol v\in \rsp{\bar V}\symbol{92}\set {\boldsymbol {0}}$, where $Z(\boldsymbol o)\subsetneq Z(\boldsymbol v)$. Then, $\boldsymbol h =\boldsymbol o - \frac{o_{\downarrow}}{v_{\uparrow}}\boldsymbol v \geq \bec 0$ and $Z(\boldsymbol v)\subseteq Z(\boldsymbol h)$. But $\bec h\in \rsp{\bar V}\cap O_+$ and we have $Z(\boldsymbol o)\subsetneq Z(\boldsymbol v)\subseteq Z(\boldsymbol h)$, which contradicts the fact that $\bec o$ is maximally-$0$ in $\rsp{\bar V}\cap O_+$.
\end{proof}
\begin{lemma}
    $O_1\cap M$ is a positive spanning set of $\rsp{\bar V} \cap O_+$.
\end{lemma}
\begin{proof}
    For the sake of contradiction, suppose that $O_1\cap M$ is not a positive spanning set and consider $\boldsymbol o\in\rsp{\bar V} \cap O_+$ to be a maximally-$0$ vector not its positive span. As $\boldsymbol o$ is not in the span of any element of $O_1\cap M$, $\boldsymbol o$ is not maximally-$0$ in $\rsp{\bar V}$. Therefore, there exists $0\neq \boldsymbol v \in \rsp{\bar V}$ where $Z(\boldsymbol o)\subsetneq Z(\boldsymbol v)$. For all non-zero elements of $\boldsymbol v$, consider $r_i=\frac{o_i}{v_i}$ and let $r=\min\set{r_i:r_i>0}$. Then, $\boldsymbol w = \boldsymbol o - r \boldsymbol v\in \rsp{\bar V} \cap O_+$, so as $Z(\boldsymbol o) \subsetneq Z(\boldsymbol w)$, $\boldsymbol w$ is in the positive span of $B$. For $w_i\neq 0$, let $s_i=\frac{o_i}{w_i}$ and $s = \min \set{s_i}$. Then, $\boldsymbol o - s \boldsymbol w \in \rsp{\bar V} \cap O_+$ and $Z(\boldsymbol o)\subsetneq Z(\boldsymbol o - s \boldsymbol w)$. Therefore, $\boldsymbol o - s \boldsymbol w$ and $\boldsymbol w$ are both elements of the positive span of $O_1\cap M$, so $\boldsymbol o$ must also be an element of the positive span of $O_1\cap M$.
\end{proof}
\begin{theorem}
The hyperplane $\set{\bec x\in \R^n : \bec a \cdot \bec x = c}$ is a facet of the polytope whose vertices are the columns of $V$ if and only if $$\bar V^T \begin{pmatrix}
    \bec a\\
    -c
\end{pmatrix}\in M.$$
\label{thm:restatement}
\end{theorem}
\begin{proof}
    By \cite{Ziegler2000}, the hyperplane $\set{\bec x : \bec a \cdot \bec x = c}$ is a facet of a polytope $P$ if and only if for all vertices $\bec v$ of $P$, $\bec a \cdot \bec v \ge c$ with equality for a maximal subset of vertices. This happens if and only if the vector
    $$V^T \bec a \ge c\bec 1$$
    with equality for a maximal subset of the entries, which is true if and only if
    $$\bec b :=\bar V^T \begin{pmatrix}
        \bec a\\
        -c
    \end{pmatrix}\ge \bec 0$$
    and is maximally-$0$ among vectors in $\rsp{\bar V}$. By Lemma \ref{lem:MequivMax0}, this is equivalent to $\bec b\in M$.
\end{proof}
\begin{corollary}
There is a one to one correspondence between elements of $O_1 \cap M$ and facets of the polytope whose vertices are the columns of $V$. 
\end{corollary}
\noindent By \cite{Ziegler2000}, every non-redundant set of inequalities bounding a polytope $P$ has exactly one inequality for every face. This means that if a bookmaker verifies that they are not giving a Dutch book by checking a fixed set of inequalities, it suffices to check the inequalities $Q(\bec b)\ge 0$ for all $\bec b\in M\cap O_1$, and every non-redundant set of inequalities they must check consists of rescaled versions of these inequalities.
\section{Merging Individually Coherent Experts} \label{sec:indCohExp}
\noindent In this section we consider a setting in which each of several experts submits a set of internally coherent credences, but the union of all expert credences is possibly incoherent. To aggregate these credences into a single coherent set of beliefs, we could use a loss function of the form (\ref{eqn:lossfunction}), but that ignores the additional information that beliefs from the same expert are coherent. Since logical inferences can be made from an expert's stated credences, the expert can express the same information in multiple ways. As long as the set of possible inferences about an expert's beliefs is the same, the loss function should be invariant under the specific information the expert shares, a property that the method for separate beliefs previously discussed does not have if naively applied here. We call this property \textit{content invariance} and discuss it further below.
\\\\
For example, if a meteorologist says that there is a $40\%$ chance of rain tomorrow and a $70\%$ chance of clouds, it is reasonable to infer that there is a $30\%$ chance of sun and a $30\%$ chance of clouds without rain. In such a situation, the meteorologist might also explicitly say the chance of it being sunny, however, because that statement does not provide new information, the loss function should be invariant under whether or not they say it.
\\\\
Recall that $\bar V$ obtained by appending a row of $1$s to $V$ and that $\rsp{\bar V}$ is the row span of $\bar V$.
\begin{definition}
    Two coherent credence bases $(V, \bec q)$ and $(V', \bec q')$ are \textit{content equivalent} if $\rsp{\bar V}=\rsp{\bar V'}$ and the credence base $((V, V'), (\bec q, \bec q'))$ is coherent.
\end{definition}
\begin{lemma}
    If $(V, \bec q)$ and $(V', \bec q')$ are content equivalent and $\bec \pi$ is a probability vector, then
    $$V\bec \pi = \bec q \iff V'\bec \pi = \bec q'.$$
    \label{lem:piSameContent}
\end{lemma}
\begin{proof}
    Since $\ker(\bar V) = \rsp{\bar V}^\perp$, the orthogonal space to the row span of $\bar V$, if $\rsp{\bar V}=\rsp{\bar V'}$, then $\ker(\bar V)=\ker(\bar V')$. By the coherence of $((V, V'), (\bec q, \bec q'))$, there is some $\bec{\tilde\pi}$ so that $V\bec{\tilde\pi}=\bec q$ and $V'\bec{\tilde\pi}=\bec q'$. Then, for any probability vector $\bec \pi$, as $\bec\pi -\bec{\tilde\pi}\in \ker(\bec 1^T)$,
    $$V\bec \pi = \bec q \iff \bec \pi - \tilde{\bec \pi}\in \ker(\bar V)\iff \bec \pi - \tilde{\bec \pi}\in \ker(\bar V')\iff V'\bec \pi = \bec q'.\qedhere$$
    
\end{proof}
\noindent Lemma \ref{lem:piSameContent} implies that content equivalence is transitive, so is an equivalence relation. 
\begin{definition}
    Let $\phi$ be a function of a credence base. $\phi$ is \textit{content invariant} if for content equivalent credence bases $\mathcal{Q}$ and $\mathcal{Q'}$ we have $\phi(\mathcal{Q})=\phi(\mathcal{Q}')$.
\end{definition}
\subsection{Content Invariance}\label{subsec:contentInvariance}
Suppose that $(V, \bec q)$ is a coherent credence base where the $n+1$ rows of $\bar V$ are linearly independent. Let $I=\rsp{\bar V}\cap \set{0, 1}^N$, which the next lemma shows is the set of events whose probabilities are linearly inferable.
\begin{lemma}
    Consider the statements
    \begin{enro}
        \item $\bec E\in I$.
        \item There is a unique belief $q$ so that the credence base $((\bec E, V), (q, \bec q))$ is coherent.
    \end{enro}
    (i) implies (ii). Moreover, if there is a probability vector $\bec \pi$ so that $V \bec \pi = \bec q$ where none of the entries of $\bec \pi$ are $0$ or $1$, then (ii) implies (i).
\end{lemma}
\begin{proof}
    Suppose (i). The belief $q=\bec E \cdot \bec \pi$ makes $((\bec E, V), (q, \bec q))$ coherent as $(V, \bec q)$ is coherent. For uniqueness, suppose we there are two coherent beliefs $q, q'$ about some event for $\bec E\in I$. Then as there is some $\bec a\in \R^n$ where $V^T \bec a = \bec E$, there are two probability vectors $\bec \pi$ and $\bec \pi'$ where $q = (V^t \bec a) \cdot \bec \pi$ and $q' = (V^t \bec a) \cdot \bec \pi'$ with $\bec q = V \bec \pi = V \bec \pi'$. This means that
    $$q=(V^t \bec a) \cdot \bec \pi = \bec a^T V \bec\pi = \bec a^T \bec q = \bec a^T V \bec\pi' = q'.$$
    Suppose not (i) and let $\bar V'$ be the matrix $\bar V$ with the vector $\bec E$ appended as the last row. Then, there is a vector $\bec v\in \ker(\bar V)\symbol{92}\ker(\bar V')$. By the assumption that no entries of $\bec \pi$ are $0$ or $1$, there is an $\epsilon>0$ so that for all $|a|<\epsilon$, $\bec \pi + a\bec v$ is also a probability vector. Then, for all such $a$, taking $q=\bec E \cdot (\bec \pi + a\bec v)$ makes $((\bec E, V), (q, \bec q))$ coherent. As $\bec v \in \ker(\bar V)\symbol{92}\ker(\bar V')$, we have $\bec E\cdot \bec v\neq 0$, so each distinct value of $a$ yields a distinct $q$.
\end{proof}
\begin{lemma}
    Let $R$ be the reduced row-echelon form of $\bar V$. Then,
    $$I=\set{R^T \boldsymbol v:\boldsymbol v \in \set{0, 1}^{n+1}, R^T \boldsymbol v \in \set{0, 1}^N}.$$
\end{lemma}
\begin{proof}
    Firstly,
$$\set{R^T \boldsymbol v:\boldsymbol v \in \set{0, 1}^{n+1}, R^T \boldsymbol v \in \set{0, 1}^N}\subseteq \rsp{\bar V}\cap\set{0, 1}^N=I.$$
Note that all of $R^T$'s columns have a leading $1$ which is the only entry in its row, so each of $\boldsymbol v$'s entries appear somewhere in $R^T\boldsymbol v$. Therefore, if $R^T\boldsymbol v \in \set{0, 1}^N$, then $\boldsymbol v \in \set{0, 1}^{n+1}$.
\end{proof}
\begin{corollary}
    The maximum size of $I$ is $2^{n+1}$.
\end{corollary}
\noindent Because $I$ is defined only in terms of the row span $\rsp{\bar V}$, it is content invariant. However, $I$ tends to be quite large even for well-tamed events, so a smaller, content invariant set might be desirable.
\\\\
Let $B\subseteq I$ be a minimal positive spanning set of $I$.
\begin{lemma}
    Any maximally-$0$ element of $I$ must be part of $B$.
\end{lemma}
\begin{proof}
    Suppose that $\bec v\in I$ is maximally-$0$ in $I$. As $B$ is a positive spanning set, we can write
    $$\bec v = \sum_{\bec b\in B} c_{\bec b} \bec b.$$
    As $\bec v\ne \bec 0$, there is some $\bec b\in B$ where $c_{\bec b}\ne \bec 0$, so $\bec b \le \bec v$ and $Z(\bec b)\subseteq Z(\bec v)$. As $\bec v$ is maximally-$0$, $Z(\bec b)= Z(\bec v)$ which means that $\bec v = \bec b$ as both are $0,1$-vectors.
\end{proof}
\begin{theorem}
    \label{thm:Bmax0}$B$ is the set of maximally-$0$ elements of $I$.
\end{theorem}
\begin{proof}
    It suffices to show that the maximally-$0$ elements of $I$ are a positive spanning set. Suppose that this were not the case. Let $\boldsymbol v$ be a maximally-$0$ vector in $I\symbol{92}\rsp{B}_+$, the maximal set positively spanned by $B$. Note that $\boldsymbol v$ is not maximally-$0$ in $I$ or it would be a member of $B$. Then, there exists $\boldsymbol w\in I\symbol{92}\set{0}$ where $Z(\boldsymbol v)\subsetneq Z(\boldsymbol w)$. Note that $Z(\boldsymbol v)\subsetneq Z(\boldsymbol v-\bec w)$, meaning $\boldsymbol w$ and $\boldsymbol v-\boldsymbol w$ are both in the positive span of $B$, so $\boldsymbol v$ is also in the positive span of $B$.
\end{proof}
\begin{lemma}
    Let $Q:I \rightarrow [0, 1]$ be the implied belief function, taking an event $\bec E$ to the unique belief $q$ that makes the credence base $((\bec E, V), (q, \bec q))$ coherent. The function $Q$ as well as the sets $I$ and $B$ are content invariant. \label{lem:I_content_invariant}
\end{lemma}
\begin{proof}
    Let $(V, \bec q)$ and $(V', \bec q')$ be two content equivalent coherent credence bases with set of inferable events $I=\rsp{\bar V} \cap \set{0, 1}^N$ and $I' = \rsp{\bar V'} \cap \set{0, 1}^N$, positive bases $B$ and $B'$ of $I$ and $I'$, and implied belief functions $Q$ and $Q'$. Then, by definition, $\rsp{\bar V} = \rsp{\bar V'}$, so $I=\rsp{\bar V} \cap \set{0, 1}^N=\rsp{\bar V'} \cap \set{0, 1}^N=I'$ is content invariant. Also, because $B$ and $B'$ are defined solely in terms of $I$ and $I'$ respectively, $B=B'$.
    \\\\
    To show that $Q$ is content invariant, if $Q(\bec E) = q$, then by coherence, there is a probability vector $\bec \pi$ so that $V\bec \pi = \bec q$ and $\bec E \cdot \bec \pi = q$. By Lemma \ref{lem:piSameContent}, we have $V' \bec \pi = \bec q'$, meaning that $\bec \pi$ certifies that $((\bec E, V'), (q, \bec q'))$ is coherent, so by uniqueness, $Q'(\bec E) = q$.
\end{proof}
    
\subsection{Loss Functions with Content Invariance}
Suppose that there are $k$ experts, the $i$th of which tells us the coherent credence base of their beliefs $\mathcal E_i = (V_i, \bec q_i)$,  with event matrix $V_i$ and belief vector $\bec q_i$ defined previously. Let $\mathcal Q$ be the combined credence base of all experts, formed by concatenating $V_1, \dots, V_k$ and $\bec q_1, \dots, \bec q_k$. We will consider loss functions of the form
$$L(\bec p, \mathcal{Q})=\sum_{i=1}^{k} \mathcal D_i(\bec p, \mathcal{E}_i)$$
for some content invariant measure of disagreement $\mathcal D_i$. In contrast to Section \ref{sec:seperateBeliefs}, the sum is over experts rather than over individual events. 
As before,
$$L^*(\mathcal{Q})=\min_{\bec p \in C(\mathtt E)}L( \mathcal{Q})=L(\bec{p^*}, \mathcal Q)$$
is a measure of the incoherence of the set of experts as a group.
\\\\
For the $i$th expert, let $I_i=\rsp{\bar V_i} \cap \set{0, 1}^N$ be the set of events for which we can infer a probability from expert $i$'s stated beliefs, let $Q_i:I_i\rightarrow [0, 1]$ be this inferred probability, and let $B_i$ be the minimal positive spanning set of $I_i$, as defined in Section \ref{subsec:contentInvariance}. Also, write $P(\bec E)$ be the entry of $\bec p$ corresponding to event $E$.
\\\\
As $Q_i$ is content invariant one approach is to let
\begin{equation}
    \mathcal D_i = D_{i, S_i} := \frac{1}{\#S_i} \sum_{\bec b\in S_i} \ell(P(\bec b), Q_i(\bec b))\label{eqn:contentinvariantloss}
\end{equation}
where $S_i\subseteq I_i$ is a content invariant set of events and $\ell$ is a dissimilarity function as in Section \ref{sec:seperateBeliefs}. As $I_i$ and $B_i$ are both content invariant, they are both choices for $S_i$.
\\\\
One possible justification for summing over $B_i$ instead of $I_i$ is that $\#I_i$ can be much larger than $\#B_i$. $B_i$ contains the maximally-$0$ elements of $I_i$, or the elements whose probabilities are smallest. As the dissimilarity functions $f$ and $f^o$ punish more for the same error when $q$ is smaller ($f(0.101, 0.001)>f(0.6, 0.5)$), summing over $B_i$ is summing over the elements of $I_i$ whose relative errors will be greatest while not considering the other terms for ease of computation.
\\\\
The normalizing coefficient $\frac{1}{\#S}$ is included in the loss function in order to give each expert equal weight, regardless of how many beliefs they express. If the measures of disagreement were not normalized, then an expert might become twice as influential on the overall loss function by expressing one additional belief because $\#I_i$ and $\#B_i$ can grow exponentially with the number of beliefs expressed.
\subsection{Removing the Symmetry between an Event and its Complement}
\noindent Looking at the dissimilarity function $f$, there are two terms: $p\ln \frac pq$ representing the expected surprise from the event occurring and $(1-p)\ln \frac{1-p}{1-q}$, representing the expected surprise from the event not occurring. However, if we sum the dissimilarity function over a minimal positive basis, then being surprised by the complement of an event happening is redundant; when one event does not happen, we are surprised instead by other events happening. Alternatively, an agent observing the world might only keep track of which events do happen, so cannot be surprised by an event not happening.
\\\\
This motivates the removal the term representing the surprise of an event not happening from the loss function. However, one needs to be careful in doing this: so far, we have required that loss functions $L(\bec p, \bec q)$ be non-negative and equal to $0$ if and only if $\bec p = \bec q$. For example, $p\ln \frac pq$ by itself is not a dissimilarity function, so the corresponding ``loss function" $L$ does not have this property in general. The following lemmas motivates two possible content invariant disagreement functions that have this property:
\begin{lemma}
    Let $(V, \bec p)$ and $(V, \bec q)$ be coherent credence bases and suppose that $V^T \bec 1=k\bec 1$ for some $k\in \mathbb N$. Then,
    $$\sum_{i=1}^n p_i \ln \frac{p_i}{q_i}\ge 0$$
    with equality if and only if $\bec p = \bec q$. \label{lem:log-sum}
\end{lemma}
\begin{proof}
    By Lemma \ref{lem:ConsistencyOnRemoval1}, since $V^T \bec 1 - k\bec 1 = \bec 0$, we have 
    $$\sum_{i=1}^n p_i - k = \sum_{i=1}^n q_i - k = 0.$$
    Therefore, $\sum_{i=1}^n p_i = k = \sum_{i=1}^n q_i$, so applying the log-sum inequality, Theorem 2.7.1 of \cite{infoTheory}, implies that
    $$\sum_{i=1}^n p_i \ln \frac{p_i}{q_i}\ge \left(\sum_{i=1}^n p_i\right) \ln \frac{\sum_{i=1}^n p_i}{\sum_{i=1}^n q_i} = 0$$
    with equality if and only if $\frac{p_i}{q_i}$ is constant. By coherence, this happens if and only if $\bec p = \bec q$.
\end{proof}
\noindent Recall from Section \ref{section:predd} that for a proper scoring rule $s(i, q)$, there is an associated dissimilarity function $\ell(p, q) = p(s(1, p) - s(1, q)) + (1-p)(s(0, p) -s(0, q))$. When using the log-loss scoring rule, this becomes $\ell = f$. Can we generalize Lemma \ref{lem:log-sum} to apply to half-dissimilarity functions
\begin{equation}\label{eqn:tilde}
    \tilde{\ell}(p, q):=ps(1, p) - ps(1, q)
\end{equation}
for a proper scoring rule $s$? No. The following lemma shows that log-loss is the only differentiable proper scoring rule for which Lemma \ref{lem:log-sum} holds.
\begin{lemma}
    Let $s(i, q)$ be a proper scoring rule that is differentiable in $q$ for $q\in (0, 1)$ with $\tilde \ell$ as defined in equation (\ref{eqn:tilde}).
    Suppose that for all coherent credence bases $(V, \bec q)$ where $V^T \bec 1=k\bec 1$ for some $k\in \mathbb N$,
    $$\bec q  = \argmin_{\bec p \in C(V)} \sum_{i=1}^n \tilde \ell(p_i, q_i)$$
    Then for some $\lambda < 0$ and $c\in \R$, 
    $${s(i, q) = \lambda(i\ln q + (1-i)\ln (1-q)) + c}.$$\label{lem:loglossunique}
\end{lemma}
\begin{proof}
    Note that by continuity, the value of $s$ when $q\in \set{0, 1}$ are determined by its values on the interval $(0, 1)$. Write 
    $$L(\bec p, \bec q) := \sum_{i=1}^n \tilde \ell(p_i, q_i)$$
    Let $V$ be the $3\times 3$ identity matrix and let $\lambda = \frac 12 s'(1, \frac 12)$ where $s'$ is the derivative of $s$ in $q$. For any $0<q < \frac 12$, consider the coherent belief vector $\bec q = (q, \frac 12, \frac 12 - q)^T$.
    By the assumption that $L(\bec p, \bec q)$ is minimized among coherent $\bec p$ when $\bec p = \bec q$, Lagrange multipliers yield that
    $$c\bec 1 = \pderiv{L}{\bec p}(\bec q, \bec q) = \begin{pmatrix}
        \lambda \\ q s'(1, q) \\ (\frac{1}{2}-q) s'(1, \frac{1}{2}-q)
    \end{pmatrix}$$
    so $q s'(1, q)=\lambda$ for all $0<q < \frac 12$. For $\frac 12 < q <1$, take $V$ to be the $2\times 2$ identity matrix and $\bec q = (q, 1 - q)$. The same argument as above shows that $q s'(1, q) = (1-q)s'(1, 1-q) = \lambda$. Therefore, for all $q\in (0, 1)$, we have $q s'(1, q) = \lambda$, so $s'(1, q) = \frac{\lambda}{q}$ and $s(1, q) = \lambda \ln q + c_1$ for some $c_1\in \R$.
    \\\\
    As $s$ is a proper scoring rule, the quantity $p s(1, q) + (1-p)s(0, q)$ is uniquely minimized in $q$ when $p=q$, so $-(1-p) s'(0, p) = ps'(1, p) =\lambda$, meaning that $s'(0, p) = \frac{-1}{1-p}$ and $s(0, p) = \lambda \ln (1-q)+c_2$ for some $c_2\in \R$. Therefore, we have
    $$s(i, q) = \lambda(i\ln q + (1-i)\ln (1-q)) + c_1+c_2.$$
    In order for $s$ to be a proper scoring rule, we must have $\lambda < 0$.
\end{proof}
\noindent Letting $O_i$ be the set of all subsets of $B_i$ whose sum is $\boldsymbol 1$ and $M_i=\sum_{S\in O_i} \#S$, we can use the dissimilarity function
\begin{equation}
    \mathcal D_i = \tilde D_{i, B_i} = \frac{1}{M_i}\sum_{S \in O_i} \sum_{\boldsymbol b \in S}\tilde \ell(P(\bec b), Q(\bec b)).\label{eqn:assymetricvariant}
\end{equation}
Lemma \ref{lem:log-sum} implies that when using $\tilde \ell(p, q) = \tilde f(p, q) := p\ln \frac pq$ or $\tilde \ell(p, q) = \tilde f^o(p, q) := q\ln \frac qp$, the measure of disagreement $\tilde D_{i, B_i}$ is minimized when $P(\bec b) = Q_i(\bec b)$ for all $\bec b \in B_i$.
\\\\
We compare the methods presented here in Appendix \ref{sec:specialcompare} as well as in Appendix \ref{sec:ex}.
\color{black}
\section*{Acknowledgements}

We thank Kenny Easwaran, Jonathan Gabor, Oliver Hopcroft, David Krueger, Yuval Peres, Suvadip Sana, Luchen Shi, and Ariel Yadin for inspiring discussions.
\pagestyle{plain}
\printbibliography
\pagestyle{headings}
\appendix
\section{Numerical Comparison between $f$ and $f^o$\label{sec:compare}}
\noindent To contrast the two settings in Section \ref{sec:seperateBeliefs}: One should use the dissimilarity function $f$ when submitting multiple forecasts based on their possibility incoherent internal beliefs (and expects these forecasts to be scored with the logarithmic proper scoring rule); its transpose $f^o$ should be used in aggregating independent sources of information, each of which gives information about one event, in order to find the most likely true probability distribution.
\\\\
Recall that $\textbf E_i\in \set{0, 1}^N$ is the vector whose $j$th entry is $1$ if $\omega_j\in E_i$ and $0$ otherwise. Additionally, $\textbf e_i$ denotes the $i$th standard basis vector and $\textbf 1$ the vector of all $1$s. For any distribution $\pi$, we associate a vector $\boldsymbol \pi$ with it, whose $j$th entry is $\pi(\omega_j)$, the probability of the $j$th atom in the Boolean algebra. 
\\\\
Table~\ref{fig:Lnumeric} contains some examples of the numerical differences between using $f$ and $f^o$ as dissimilarity functions.
\begin{table}[h]
\begin{center}
    \begin{tabular}{|m{12em} | m{10em}| m{9em} |}
         \hline
         Situation & \multicolumn{2}{c|}{$\bec{p^*}$} \\
         \cline{2-3}& $\ell=f$ & $\ell=f^o$\\
         \hline
         Exactly one event occurs:  & &
         \\$\boldsymbol E_1=\boldsymbol e_1, q_1 = 0.1$ & $p_1=0.01$&$p_1=0.04$
         \\$\boldsymbol E_2=\boldsymbol e_2, q_2 = 0.6$ & $p_2=0.11$&$p_2=0.30$
         \\$\boldsymbol E_3=\boldsymbol e_3, q_3 = 0.99$ & $p_3=0.89$&$p_3=0.66$
         \\\hline
         All estimates are for the same event:  & &
         \\${\boldsymbol E_1 = \boldsymbol E_2=\boldsymbol E_3=(1 \, 0)^T}$ & &
         \\ $q_1 = 0.1, q_2=0.3, q_3=0.5$ &$p_1=p_2=p_3=0.27$ & $p_1=p_2=p_3=0.3$
         \\\hline
         $\boldsymbol E_1 = (1 \,1 \,0 \,0)^T$, $q_1=0.99$ &$p_1=0.87$ & $p_1=0.73$
         \\
         $\boldsymbol E_2 = (1 \,0 \,1 \,0)^T$, $q_2=0.5$ &$p_2=0.40$ & $p_2=0.47$
         \\
         $\boldsymbol E_3 = (1 \, 0 \,0 \,0)^T$, $q_3=0.1$ &$p_3=0.27$ & $p_3=0.20$
         \\
         $\boldsymbol E_4 = (0 \,1 \,0 \,0)^T$, $q_4=0.4$ &$p_4=0.60$ & $p_4=0.53$
         \\
         $\boldsymbol E_5 = (0 \,0 \,1 \,0)^T$, $q_5=0.4$ &$p_5=0.13$ & $p_5=0.27$
         \\\hline
         \end{tabular}
         \caption{Three examples of correcting incoherent probabilities with $f$ and $f^o$. By Theorem \ref{punique}, $\bec{p^*}$ is unique.}
         \label{fig:Lnumeric}
\end{center}
\end{table}
\begin{figure}[h]
     \centering
     \begin{subfigure}[b]{0.4\textwidth}
         \centering
         \includegraphics[width=\textwidth]{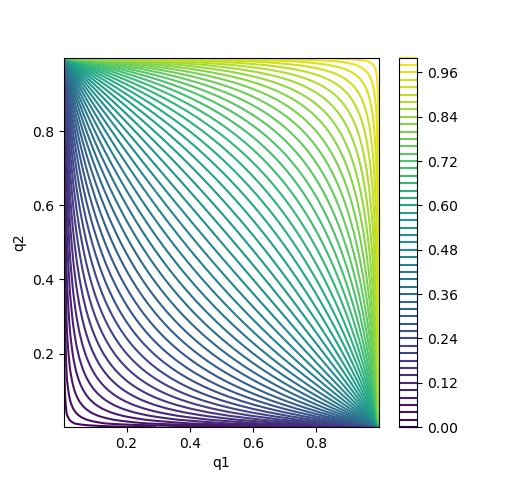}
         \caption{Using $\ell=f$}
         \label{fig:p:h=f}
     \end{subfigure}
     \hfill
     \begin{subfigure}[b]{0.4\textwidth}
         \centering
         \includegraphics[width=\textwidth]{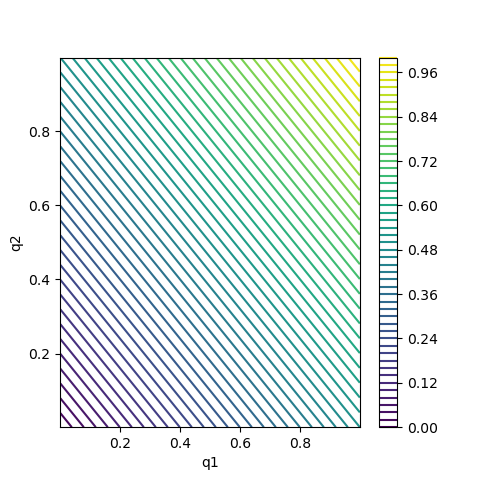}
         \caption{Using $\ell=f^o$}
         \label{fig:p:h=fo}
     \end{subfigure}
     \hfill
        \caption{Contour plots of $p^*(\bec q)$ when $E_1=E_2$ are the events being estimated. Contour distance is 0.02.
        \label{fig:p*comparisons}
        }
\end{figure}
\\\\
The differences in $f$ and $f^o$ also leads the function $L^*$ to behave differently when using the two functions, as illustrated in Figure \ref{fig:L*comparisons}, where $E_1$ and  $E_2=E_1^c$ are the events being estimated.
\begin{figure}[h]
     \centering
     \begin{subfigure}[b]{0.4\textwidth}
         \centering
         \includegraphics[width=\textwidth]{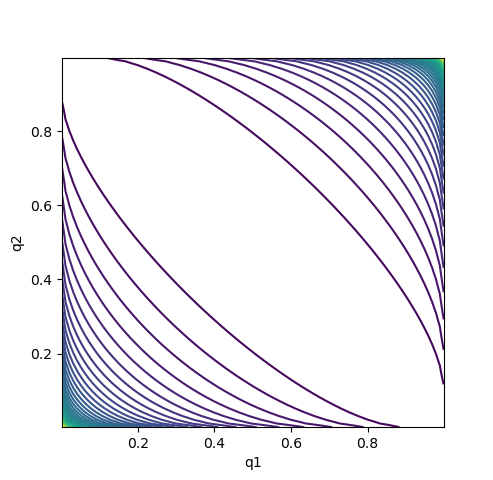}
         \caption{Using $\ell=f$}
         \label{fig:L:h=f}
     \end{subfigure}
     \hfill
     \begin{subfigure}[b]{0.4\textwidth}
         \centering
         \includegraphics[width=\textwidth]{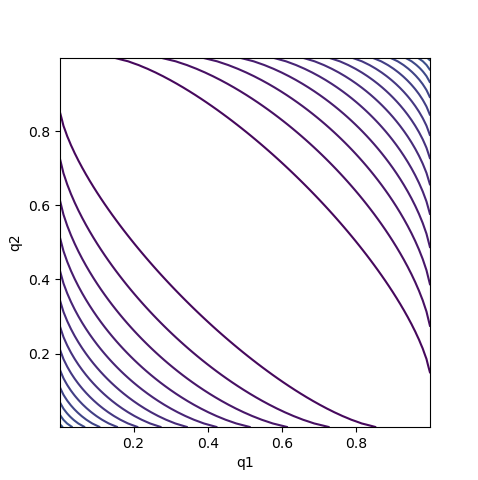}
         \caption{Using $\ell=f^o$}
         \label{fig:L:h=fo}
     \end{subfigure}
     \hfill
        \caption{Contour plots of $L^*(\bec q)$ when $E_1$ and $ E_2=E_1^c$ are the events being estimated. Near the line $q_1+q_2=1$, both functions are approximately quadratic and can be approximated by $\frac{(q_1+q_2-1)^2}{(1-q_1+q_2)(1+q_1-q_2)}$. Also, when using $\ell=f$, $L^*$ is not bounded, while it is when using $\ell=f^o$. Contour distance is 0.1.}
        \label{fig:L*comparisons}
\end{figure}
\section{Computing \texorpdfstring{$\bec{p^*}$}{p*} and \texorpdfstring{$L^*$}{L*}\label{sec:findpstar}}
\noindent One way to find $L^*$ and $\boldsymbol{p^*}$ numerically is to define a function that, given a vector $\boldsymbol \pi$ and the $n\times N$ matrix $V$ whose $i$th row is $\textbf E_i$, and the credences $\boldsymbol q$, returns  $L(V\boldsymbol\pi, \boldsymbol q)$. Then, one can use a gradient descent algorithm to minimize $L(V\boldsymbol\pi, \boldsymbol q)$ under the conditions that the entries of $\boldsymbol \pi$ are nonnegative and sum to $1$. This algorithm is illustrated below in psuedocode.\\
\begin{algorithmic}
\State $\bec {\pi_0} \gets \frac{1}{N} \bec{1}$ \Comment{Initial guess is that all atoms have equal probability}
\State \Return gradientDescent($\bec \pi \mapsto L(V\boldsymbol\pi, \boldsymbol q)$,\Comment{The loss function in terms of $\bec \pi$}\\ 
    \hspace{1.47in} $\bec{\pi_0}$, \Comment{Initial guess of $\bec \pi$}\\
    \hspace{1.47in} $\bec \pi \cdot \bec 1 = 1$, $\bec 0 \le \bec \pi \le \bec 1$) \Comment{Conditions $\bec \pi$ must satisfy}\\
\end{algorithmic}
Computing $\bec{p^*}(\bec q)$ to within an accuracy of $\epsilon$ takes $O\left(nN\log\left(\frac{N}{\epsilon}\right)\right)$ time \cite{bubeck_convex_2015}. Python code to compute $\bec{p^*}$ and $L^*$ can be found on \url{https://github.com/scim142/quantifying_coherence}.
\section{Holder Continuity of $\bec{p^*}$ and $L^*$\label{sec:continuity}}
\noindent Let $||\cdot ||$ denote the Euclidean norm.
\begin{lemma}
If $\ell$ is (strictly) convex/$n$-times differentiable/Lipschitz in $p$/$q$/both, then $L$ will also be. \label{lem:2.1}
\end{lemma}
\begin{proof}
    (convexity in $p$) Let $\ell_i=w_i\ell(\textbf p \cdot \textbf e_i, \textbf q \cdot \textbf e_i)$. Then, for $t\in[0, 1]$, $$\ell_i(t\textbf p + (1-t)\textbf p', \textbf q)=w_i \ell(tp_i+(1-t)p_i', q_i)$$
    $$\le w_i t\ell(p_i, q_i)+w_i(1-t)\ell(p_i', q_i)=t\ell_i(\textbf p, \textbf q)+(1-t)\ell_i(\textbf p', \textbf q)$$
    so $\ell_i$ is convex. As $L(\textbf p, \textbf q) = \sum_{i=1}^{n} \ell_i(\textbf p, \textbf q)$ is the sum of convex functions, $L$ is convex.
\end{proof}
\begin{lemma}
\label{Lqsim}
If $\ell$ is Lipchitz in $q$ in some region $D\subseteq [0, 1]^2$, then $L^*$ will be Lipchitz on the set $\bar{D}:=\set{\textbf q\in [0, 1]^n:(\boldsymbol{p^*}(\textbf q), \textbf q)\in D}$ with Lipschitz constant equal to that of $L$ in $\boldsymbol q$.
\end{lemma}
\begin{proof}
    
    Let $\boldsymbol q, \boldsymbol{q'}\in \bar D$ and consider the case when $L^*(\textbf q)\le L^*(\boldsymbol{q'})$. Let $\boldsymbol \delta = \boldsymbol{q'} - \boldsymbol q$ and $k$ be the Lipschitz constant of $L$. Then,
    $$L^*(\textbf q)\le L^*(\textbf q + \boldsymbol \delta)\le L(\boldsymbol{p^*}(\boldsymbol q), \boldsymbol q + \boldsymbol \delta)\le L^*(\boldsymbol q) + k ||\boldsymbol \delta||$$
    If $L^*(\boldsymbol q)> L^*(\boldsymbol{q'})$, then the above shows that $L^*(\boldsymbol q)-k||\boldsymbol \delta||\le L^*(\boldsymbol{q'})$. Therefore, regardless of the relative sizes of $L^*(\boldsymbol q)$ and $L^*(\boldsymbol{q'})$, we have
    $$L^*(\boldsymbol q)-k||\boldsymbol \delta||\le L^*(\boldsymbol{q'})\le L^*(\boldsymbol q)+k||\boldsymbol \delta||.\qedhere$$
\end{proof}

\begin{lemma}
    If $\ell$ is strictly convex and twice differentiable in $(0, 1)^n$, then $\boldsymbol{p^*}(\boldsymbol q)$ is continuous. \label{lem:p*continuous}
\end{lemma}
\begin{proof}
    Fix any $\bec q \in (0, 1)^n$ and $k_q$, $k_p$ greater than the Lipschitz constants of $L$ with respect to $\bec q$ and $\bec p$ respectively in the neighborhood of $(\bec{p^*}(\bec q), \bec q)$. Let $D$ be a closed region in which $L$ is Lipschitz in $\boldsymbol{p}$ with Lipschitz constant $k_p$ and Lipshcitz in $\boldsymbol{q}$ with Lipschitz constant $k_q$ and $\bar D$ be a closed region contained in $\set{\boldsymbol q':(\boldsymbol{p^*}(\boldsymbol{q'}), \boldsymbol{q'})\in D}$ that contains $(\bec{p^*}(\bec q), \bec q)$.\\\\
    Let $\boldsymbol{q_1}, \boldsymbol{q_2}, \dots$ be a sequence in $\bar D$ with limit $\boldsymbol q$. As $(\boldsymbol{p^*}(\boldsymbol a), \boldsymbol{a})\in D$ for all $\boldsymbol a\in \bar D$, and $D$ is sequentially compact, the sequence $\boldsymbol{p^*}(\boldsymbol{q_1}), \boldsymbol{p^*}(\boldsymbol{q_2}), \dots$ has a subsequence with a limit. Let $\boldsymbol{p'}$ be one of those limit and $\boldsymbol{a_1}, \boldsymbol{a_2}, \dots$ be a subsequence of $\boldsymbol{q_1}, \boldsymbol{q_2}, \dots$ where the limit of $\boldsymbol{p^*}(\boldsymbol{a_1}), \boldsymbol{p^*}(\boldsymbol{a_2}), \dots$ is $\boldsymbol{p'}$. Then, for any $\epsilon >0$, there exists $N$ where for all $n>N$, $||\boldsymbol{a_n}-\boldsymbol q||<\epsilon$ and $||\boldsymbol{p^*}(\boldsymbol{a_n})-\boldsymbol{p'}||<\epsilon$. Then, by Lemma~\ref{Lqsim},
    $$||L(\boldsymbol{p^*}(\boldsymbol{a_n}), \boldsymbol{a_n})-L(\boldsymbol{p^*}(\boldsymbol{q}), \boldsymbol{a})||\le k_q\epsilon.$$
    Using the fact that $L$ is Lipschitz in $\boldsymbol p$ and $\boldsymbol q$, by the triangle inequality
    $$||L(\boldsymbol{p'}, \boldsymbol{q})-L(\boldsymbol{p^*}(\boldsymbol{q}), \boldsymbol{a})||-k_p\epsilon-k_q\epsilon\le ||L(\boldsymbol{p^*}(\boldsymbol{a_n}), \boldsymbol{a_n})-L(\boldsymbol{p^*}(\boldsymbol{q}), \boldsymbol{a})||$$
    $$||L(\boldsymbol{p'}, \boldsymbol{q})-L(\boldsymbol{p^*}(\boldsymbol{q}), \boldsymbol{a})||\le (2k_q+k_p)\epsilon.$$
    As this is true for all $\epsilon>0$, we must have $||L(\boldsymbol{p'}, \boldsymbol{q})-L(\boldsymbol{p^*}(\boldsymbol{q}), \boldsymbol{a})||=0$, so by Theorem \ref{punique}, $\boldsymbol{p'}=\boldsymbol{p^*}(\boldsymbol{q})$ and $\boldsymbol{p^*}$ is continuous at $\boldsymbol q$.
    
\end{proof}
\noindent Let $\nabla L(\bec p, \bec q)$ be the gradient of $L$ with respect to $\bec p$ at $(\bec p, \bec q)$.
\begin{theorem}
\label{holder}
    If $\ell$ is strictly convex and twice differentiable in $(0, 1)^n$, then $\boldsymbol{p^*}(\boldsymbol q)$ is Holder-$\frac 12$. \label{thm:holder}
\end{theorem}
\begin{proof}
    Fix any $\bec q \in (0, 1)^n)$ and let $\lambda$ be the least eigenvalue of the Hessian $H$ of $L$ with respect to $\boldsymbol p$ at $(\boldsymbol{p^*}(\boldsymbol q), \boldsymbol q)$, let $D$ be some open region over which $L$ is Lipschitz that contains  $(\boldsymbol{p^*}(\boldsymbol q), \boldsymbol q)$, and let $\bar{D}=\set{\textbf q\in [0, 1]^n:(\boldsymbol{p^*}(\textbf q), \textbf q)\in D}$. Additionally, let $k$ be the Lipschitz constant of $L$ with respect to $\boldsymbol q$ on $\bar D$. Then, first of all, by Lemma~\ref{Lqsim}, for all $\boldsymbol \delta$ where $\boldsymbol q + \boldsymbol \delta \in \bar D$,
    $$|L(\boldsymbol{p^*}(\boldsymbol q+\boldsymbol \delta), \boldsymbol q+\boldsymbol \delta)-L(\boldsymbol{p^*}(\boldsymbol q), \boldsymbol q)|\le k ||\boldsymbol \delta||.$$
    Then, by the triangle inequality, as $L$ is Lipschitz in $\boldsymbol q$
    $$|L(\boldsymbol{p^*}(\boldsymbol q+\boldsymbol \delta), \boldsymbol q)-L(\boldsymbol{p^*}(\boldsymbol q), \boldsymbol q)|$$
    $$\le |L(\boldsymbol{p^*}(\boldsymbol q+\boldsymbol\delta), \boldsymbol q)-L(\boldsymbol{p^*}(\boldsymbol q+\boldsymbol \delta), \boldsymbol q+\boldsymbol \delta)|+|L(\boldsymbol{p^*}(\boldsymbol q+\boldsymbol \delta), \boldsymbol q+\boldsymbol \delta)-L(\boldsymbol{p^*}(\boldsymbol q), \boldsymbol q)|\le2 k ||\boldsymbol \delta||.$$
    So,
    $$L(\boldsymbol{p^*}(\boldsymbol q+\boldsymbol \delta), \boldsymbol q)-L(\boldsymbol{p^*}(\boldsymbol q), \boldsymbol q)\le 2k||\boldsymbol \delta||.$$
    Also, as $L$ is twice differentiable in $\boldsymbol p$ and $ \boldsymbol{p^*}$ is continuous at $\boldsymbol q$ by Lemma \ref{lem:p*continuous},
    $$0=\lim_{\boldsymbol \delta \rightarrow \boldsymbol 0}\frac{L(\boldsymbol{p^*}(\boldsymbol q+\boldsymbol \delta), \boldsymbol q)-L(\boldsymbol{p^*}(\boldsymbol q), \boldsymbol q)-\nabla L(\bec{p^*}(\bec q), \bec q)\cdot \Delta \bec{p^*}-\Delta \boldsymbol{p^{*}}^T H \Delta \boldsymbol{p^*}}{||\Delta \boldsymbol{p^*}||^2}$$
    where $\Delta \boldsymbol{p^*}=\boldsymbol{p^*}(\boldsymbol q+\boldsymbol \delta)-\boldsymbol{p^*}(\boldsymbol q)$ and $\nabla L(\bec{p^*}(\bec q), \bec q)\cdot \Delta \bec{p^*}\ge 0$ as in Theorem \ref{punique}.
    Then, for any $0<\epsilon <\lambda$, for there exists $\Delta > 0$ where for $0<||\boldsymbol \delta||<\Delta$
    $$-\epsilon<\frac{L(\boldsymbol{p^*}(\boldsymbol q+\boldsymbol \delta), \boldsymbol q)-L(\boldsymbol{p^*}(\boldsymbol q), \boldsymbol q)-\nabla L(\bec{p^*}(\bec q), \bec q)\cdot \Delta \bec{p^*}-\Delta \boldsymbol p^{*T} H \Delta \boldsymbol{p^*}}{||\Delta \boldsymbol{p^*}||^2}$$
    $$\le\frac{2k||\boldsymbol \delta||-\lambda ||\Delta \boldsymbol{p^*}||^2}{||\Delta \boldsymbol{p^*}||^2}.$$
    So
    $$\frac{||\Delta \boldsymbol p^*||}{\sqrt{||\boldsymbol\delta||}}\le\sqrt{\frac{2k}{\lambda-\epsilon}}.$$
    And $\boldsymbol{p^*}$ is Holder-$\frac{1}{2}$ in the open ball of radius $\Delta$ at $\boldsymbol q$.
\end{proof}
\noindent Let $\nabla_{\boldsymbol q} L(\boldsymbol{p^*}(\boldsymbol q), \boldsymbol q)$ be the gradient of $L(\boldsymbol p, \boldsymbol q)$ as $\boldsymbol q$ is changed at the point $(\boldsymbol{p^*}(\boldsymbol q), \boldsymbol q)$.
\begin{theorem} \label{thm:holderImpliesDiff}
        Suppose there is some open region $\bar D\subseteq(0, 1)^n$ where $\boldsymbol{p^*}$ is Holder-$\frac{1+\alpha}{2}$ for some $\alpha>0$, and $\ell$ is twice differentiable at $(\boldsymbol{p^*}(\boldsymbol q), \boldsymbol q)$ for $\boldsymbol q \in \bar D$. Then, $L^*$ is differentiable in $\bar D$ with derivative $\nabla L^*(\boldsymbol q)=\nabla_{\boldsymbol q} L(\boldsymbol{p^*}(\boldsymbol q), \boldsymbol q)$.
\end{theorem}
\begin{proof}
Let $k$ be the Holder-$\frac{1+\alpha}{2}$ coefficient of $p^*$ and $\lambda$ the least eigenvalue of the Hessian $H$ of $L$ with respect to $\boldsymbol p$ at $(\boldsymbol{p^*}(\boldsymbol q), \boldsymbol q)$. Then, consider
$$\lim_{\boldsymbol \delta \rightarrow \boldsymbol 0}\frac{L(\boldsymbol{p^*}(\boldsymbol q + \boldsymbol \delta), \boldsymbol q + \boldsymbol \delta) -L(\boldsymbol{p^*}(\boldsymbol q + \boldsymbol \delta), \boldsymbol q) -L(\boldsymbol{p^*}(\boldsymbol q), \boldsymbol q + \boldsymbol \delta) +L(\boldsymbol{p^*}(\boldsymbol q), \boldsymbol q)}{||\boldsymbol \delta||}$$
$$=\lim_{\boldsymbol \delta \rightarrow \boldsymbol 0}\nabla_{\boldsymbol q} L(\boldsymbol{p^*}(\boldsymbol q + \boldsymbol \delta), \boldsymbol q) - \nabla_{\boldsymbol q} L(\boldsymbol{p^*}(\boldsymbol q), \boldsymbol q)=0$$
as $\boldsymbol{p^*}$ and $\nabla_{\boldsymbol q} L$ are both continuous.
Then,
$$\lim_{\boldsymbol \delta \rightarrow \boldsymbol 0}\frac{L^*(\boldsymbol q + \boldsymbol\delta)-L^*(\boldsymbol q)-\nabla_{\boldsymbol q} L(\boldsymbol{p^*}(\boldsymbol q), \boldsymbol q)}{||\boldsymbol \delta||}$$
$$=\lim_{\boldsymbol \delta \rightarrow \boldsymbol 0}\frac{L(\boldsymbol{p^*}(\boldsymbol q + \boldsymbol \delta), \boldsymbol q) -L(\boldsymbol{p^*}(\boldsymbol q), \boldsymbol q)+L(\boldsymbol{p^*}(\boldsymbol q), \boldsymbol q + \boldsymbol \delta) -L(\boldsymbol{p^*}(\boldsymbol q), \boldsymbol q)-\nabla_{\boldsymbol q} L(\boldsymbol{p^*}(\boldsymbol q), \boldsymbol q)}{||\boldsymbol \delta||}$$
$$=\lim_{\boldsymbol \delta \rightarrow \boldsymbol 0}\frac{L(\boldsymbol{p^*}(\boldsymbol q + \boldsymbol \delta), \boldsymbol q) -L(\boldsymbol{p^*}(\boldsymbol q), \boldsymbol q)}{||\boldsymbol \delta||}$$
$$=\lim_{\boldsymbol \delta \rightarrow \boldsymbol 0}\frac{L(\boldsymbol{p^*}(\boldsymbol q + \boldsymbol \delta), \boldsymbol q) -L(\boldsymbol{p^*}(\boldsymbol q), \boldsymbol q)}{||\boldsymbol{p^*}(\boldsymbol v + \boldsymbol \delta)-\boldsymbol{p^*}(\boldsymbol v)||^2}\frac{||\boldsymbol{p^*}(\boldsymbol v + \boldsymbol \delta)-\boldsymbol{p^*}(\boldsymbol v)||^2}{||\boldsymbol \delta||}$$
$$\le \lim_{\boldsymbol \delta \rightarrow \boldsymbol 0}\lambda \frac{||\boldsymbol{p^*}(\boldsymbol v + \boldsymbol \delta)-\boldsymbol{p^*}(\boldsymbol v)||^2}{||\boldsymbol \delta||}=\lim_{\boldsymbol \delta \rightarrow \boldsymbol 0}\lambda ||\boldsymbol \delta||^{2\alpha}\frac{||\boldsymbol{p^*}(\boldsymbol v + \boldsymbol \delta)-\boldsymbol{p^*}(\boldsymbol v)||^2}{||\boldsymbol \delta||^{1+2\alpha}}\le\lim_{\boldsymbol \delta \rightarrow \boldsymbol 0}\lambda k ||\boldsymbol\delta||^{2\alpha}=0$$
Note that by the minimum property of $L^*$
$$0\le\lim_{\boldsymbol \delta \rightarrow \boldsymbol 0}\frac{L(\boldsymbol{p^*}(\boldsymbol q + \boldsymbol \delta), \boldsymbol q) -L(\boldsymbol{p^*}(\boldsymbol q), \boldsymbol q)}{||\boldsymbol \delta||}$$
So,
$$\lim_{\boldsymbol \delta \rightarrow \boldsymbol 0}\frac{L^*(\boldsymbol q + \boldsymbol\delta)-L^*(\boldsymbol q)-\nabla_q L(\boldsymbol{p^*}(\boldsymbol q), \boldsymbol q)}{||\boldsymbol \delta||}=0$$
and $\nabla L^*(\boldsymbol q)=\nabla_{\boldsymbol q} L(\boldsymbol{p^*}(\boldsymbol q), \boldsymbol q)$.
\end{proof}
\section{Proofs of Section \ref{subsec:singleEvent}}
\begin{lemma}
    When using dissimilarity function $\ell(p, q)=2(p-q)^2$, the loss function becomes
    $$L^*(q_E, q_{E^c}) = (q_{E} + q_{E^c} - 1)^2.\label{lem:burns_loss}$$
\end{lemma}
\begin{proof}
    Lagrange multipliers yield $p^*_E=\frac{q_E + 1 - q_{E^c}}{2}=1-p^*_{E^c}$. Therefore
    \begin{align*}
        L^*(q_{E}, q_{E^c})&=L((p_{E}, p_{E^c}), (q_{E}, q_{E^c}))\\
        &=2(p_E - q_{E})^2+2(p_{E^c} - q_{E^c})^2\\
        &=(q_{E} + q_{E^c} - 1)^2.\qedhere
    \end{align*}
\end{proof}
\begin{lemma}
    When using dissimilarity function 
    $$\ell(p, q)=f(p, q)=p\ln \frac pq + (1-p)\ln \frac{1-p}{1-q}$$
    the loss function becomes
    $$L^*(q_E, q_{E^c}) = \frac{(q_E + q_{E^c} - 1)^2}{(1-q_E+q_{E^c})(q_E+1-q_{E^c})}+O((q_E + q_{E^c} - 1)^3).\label{lem:burns_f}$$
\end{lemma}
\begin{proof}
    Lagrange multipliers yield that the closest coherent probability $p^*_{E}$ satisfies $\frac{p^*_{E}}{1-p^*_{E}}=\left(\frac{q_E(1-q_{E^c})}{(1-q_{E})q_{E^c}}\right)^{1/2}$, giving the loss
\begin{align*}
    L^*(q_E, q_{E^c})&=2p^*_{E} \ln \frac{p^*_E}{1-p^*_E} - p^*_E\ln\frac{q_E}{1-q_E} - p^*_E\ln\frac{q_{E_c}}{1-q_{E^c}}+\ln \frac{(1-p^*_E)^2}{(1-q_E)q_{E^c}}\\
    &=\ln \frac{(1-p^*_E)^2}{(1-q_E)q_{E^c}}.
\end{align*}
Letting $o_{E}= \frac{q_E}{1-q_E}$ and $o'_E=\frac{1-q_{E^c}}{q_{E^c}}$ be the odds of $E$ derived from the two probes with $q=q_E$ and $q'=1-q_{E^c}$ the probabilities derived from the probes and Taylor expanding in $o$ about $o=o'$ yields:
\begin{align*}
    L^*(q_E, q_{E^c})&=\ln \frac{(1+o)(1+o')}{(1+\sqrt{oo'})^2}\\
    &\approx \frac{1}{2}(o-o')^2 \deriv{^2}{o^2}\ln \frac{(1+o)(1+o')}{(1+\sqrt{oo'})^2}+O((o-o')^3)\\
&=\frac{1}{2}(o-o')^2\left(-\frac{1}{(1+o')^2}+\frac{1+2o'}{2o'(1+o')^2}\right)+O((o-o')^3)\\
&=\frac{1}{4o'}(o-o')^2+O((o-o')^3)\\
&=\frac{(q-q')^2}{4q'(1-q')}+O((q-q')^3)\\
&\approx \frac{(q - q')^2}{(q+q')(2-q-q')}+(q-q')^2O(q-q')+O((q-q')^3)\\
&=\frac{(q - q')^2}{(q+q')(2-q-q')}+O((q - q')^3). \qedhere
\end{align*}
\end{proof}
\begin{lemma}
    When using dissimilarity function 
    $$\ell(p, q)=f^o(p, q)=q\ln \frac qp + (1-q)\ln \frac{1-q}{1-p}$$
    the loss function becomes
    $$L^*(q_E, q_{E^c}) = \frac{(q_E + q_{E^c} - 1)^2}{(1-q_E+q_{E^c})(q_E+1-q_{E^c})}+O((q_E + q_{E^c} - 1)^3).$$\label{lem:burns_fo}
\end{lemma}
\begin{proof}
    Again letting $q=q_E$ and $q'=1-q_{E^c}$, Lagrange multipliers show that $p^*_E = \frac{q+q'}{2}$, which we will denote by $p$ for ease of notation. Taking the derivative of $L^*$ in $q'$ gives
    $$\pderiv{}{q'}L^*(q, q') = \ln\frac{q'}{1-q'}-\ln \frac{p}{1-p}$$
    and
    $$\pderiv{^2}{q'^2}L^*(q, q') = \frac{1}{q'(1-q')}-\frac{1}{2p(1-p)}.$$
    Then, Taylor expanding around $q=q'$ in $q'$ gives
    \begin{align*}
        L^*(q_E, q_{E^c}) &\approx \frac{1}{2}(q-q')^2\pderiv{^2}{q'^2}L^*(q, q') + O((q-q')^3)\\
        &=\frac{(q-q')^2}{2q'(1-q')}-\frac{q-q')^2}{4p(1-p)}+ O((q-q')^3\\
        &\approx \frac{(q-q')^2}{4p(1-p)}+(q-q')^2O((q-q'))+ O((q-q')^3\\
        &=\frac{(q-q')^2}{(q+q')(2-q-q')} + O((q-q')^3).\qedhere
    \end{align*}
\end{proof}
\section{Comparison of Expert Aggregation Methods\label{sec:specialcompare}}
\noindent Here, we will look at the following scenario and consider what various methods from Section \ref{sec:indCohExp} give us as estimates of the probabilities of various events:
\\\\
We take $N=4$ and consider two experts. The first expert submits probability estimates for two events
$$V_1=\begin{pmatrix}
    1 & 1 & 0 & 0\\
    1 & 1 & 1 & 0
\end{pmatrix}, \boldsymbol{q}_1=\begin{pmatrix}
    0.5\\0.9
\end{pmatrix}$$
and the second expert submits estimates for three events
$$V_2=\begin{pmatrix}
    1 & 1 & 0 & 0\\
    0 & 1 & 1 & 0\\
    0 & 0 & 0 & 1
\end{pmatrix}, \boldsymbol{q}_2=\begin{pmatrix}
    0.3\\0.2\\0.6
\end{pmatrix}.$$
Table \ref{table:assymetriccomparison} summarizes the aggregated beliefs $\bec p^*$ using four different methods and two different dissimilarity functions.\color{black}

\begin{table}[H]
\begin{center}
    \begin{tabular}{|p{15em} |m{2em}|m{2em}| m{6em} | m{6em} |}
         \hline
         Aggregation Method & $\bec q_1$& $\bec q_2$ &$\bec p^*$ with $\ell=f$ & $\bec p^*$ with $\ell=f^o$
        \\\hline
            Sum over stated beliefs only (not content invariant):
            \newline $\mathcal D_i = D_{i, \mathtt E_i}$&$\begin{matrix}
            0.5 \\ 0.9 \\ - \\ -
        \end{matrix}$ & $\begin{matrix}
            0.3 \\ - \\ 0.2 \\ 0.6
        \end{matrix}$
        &$\begin{matrix}
            0.43 \\ 0.68 \\ 0.25 \\ 0.32
        \end{matrix}$
        &$\begin{matrix}
            0.41 \\ 0.64 \\ 0.22 \\ 0.36
        \end{matrix}$\\
        \hline
        Sum over all inferable beliefs:
        \newline
        $\mathcal D_i=D_{i, I_i}$ &$\begin{matrix}             0.5 \\ 0.9 \\ - \\ -         \end{matrix}$ & $\begin{matrix}             0.3 \\ - \\ 0.2 \\ 0.6         \end{matrix}$& $\begin{matrix}
            0.46 \\ 0.73 \\ 0.46 \\ 0.27
        \end{matrix}$
        &$\begin{matrix}
            0.43 \\ 0.71 \\ 0.47 \\ 0.29
        \end{matrix}$\\
        \hline
        Sum over the minimal positive basis:
        \newline
        $\mathcal D_i= D_{i, B_i}$ &$\begin{matrix}             0.5 \\ 0.9 \\ - \\ -         \end{matrix}$ & $\begin{matrix}             0.3 \\ - \\ 0.2 \\ 0.6         \end{matrix}$& $\begin{matrix}
            0.46 \\ 0.72 \\ 0.42 \\ 0.28
        \end{matrix}$&$\begin{matrix}
            0.42 \\ 0.69 \\ 0.41 \\ 0.31
        \end{matrix}$\\
        \hline
        Sum over the minimal positive basis, asymmetric variant: 
        \newline $\mathcal D_i=\tilde D_{i, B_i}$ &$\begin{matrix}             0.5 \\ 0.9 \\ - \\ -         \end{matrix}$ & $\begin{matrix}             0.3 \\ - \\ 0.2 \\ 0.6         \end{matrix}$& $\begin{matrix}
            0.48 \\ 0.74 \\ 0.42 \\ 0.26
        \end{matrix}$&$\begin{matrix}
            0.41 \\ 0.69 \\ 0.41 \\ 0.31
        \end{matrix}$\\
        \hline
         \end{tabular}

\color{black}
\caption{Minimizer of $L$ for various loss functions, of four possible forms defined in equations (\ref{eqn:contentinvariantloss}) and (\ref{eqn:assymetricvariant}), three of which are content invariant, each using one of two possible dissimilarity functions. 
$\bec p^*$ is the nearest coherent vector of beliefs. Observe that all three content invariant aggregation methods give similar estimates $\bec p^*$ provided they use the same loss function.}\label{table:assymetriccomparison}
\end{center}
\end{table}
\noindent It should be noted that the first method assigns a probability of $0$ to atom $\omega_2$, but that it is given non-zero probability by all the content invariant methods. This is because the event $\bec e_2$ is not present in the sum over stated beliefs only, but is present in all the other sums.

\section{Extended Example: Masked Letters}\label{sec:ex}
\noindent We illustrate the methods of Section \ref{sec:indCohExp} in the following scenario: given a word with one masked letter, we seek to predict the masked letter by combining two heuristics, which will play the role of experts in Section \ref{sec:indCohExp}. One heuristic uses the preceding letters to make predictions and the other heuristic uses the succeeding letters to make its predictions.
\\\\For example, given the word EM*IL, with the third letter masked, the first expert would use the 3-gram EM* to make its predictions $Q_1(\text{A})=0.16, Q_1(\text{B})=0.08, \dots$. The second heuristic would use the 3-gram *IL to make the predictions $Q_2(\text{A})=0.32, Q_2(\text{B})=0.27, \dots$.
\\\\
If each heuristic is derived from counting 3-gram frequencies in a corpus, then it will be internally coherent, so we use a content invariant loss function as outlined in Section \ref{sec:indCohExp}. As $\#I_1=\#I_2=2^{26}$, the size of the set of events whose probabilities can be inferred, is much too large, one should either sum over the minimal spanning sets $B_i$ or sum over subsets of the minimal spanning sets that add to $\bec 1$. We will outline the calculations using dissimilarity function $\ell=f$ and compare the results with $\ell = f^o$.
\subsection{\texorpdfstring{Method 1: Summing over a Positive Basis}{Method 1: Summing over a Positive Basis}}
Here we consider the loss function determined by
$$\mathcal D_i=D_{i, B_i} = \frac{1}{B_i}\sum_{\boldsymbol b \in B_i}\ell(P(\boldsymbol b), Q_i(\boldsymbol b))$$
Recall that $M_i$ is the normalizing constant, equal to the number of terms in the summand. Additionally, in this situation, the heuristics' predictions are exactly the minimal spanning set. So
$$L(\bec p, (\bec q_1, \bec q_2))=L(V\bec \pi, (\bec q_1,\bec q_2))=\frac{1}{26}\sum_{\alpha\in \set{\text{A}, \text{B}, \dots}} \ell(\pi_\alpha, q_{1, \alpha})+\ell(\pi_\alpha, q_{2, \alpha}).$$
Then, if all entries in $\bec \pi$ are strictly positive (which they will be for $f$ and $f^o$ as long as all entries in $\bec q_1$ and $\bec q_2$ are non-zero), by Lagrange multipliers as $\sum_{\alpha\in \set{\text{A}, \text{B}, \dots}}\pi_\alpha=1$, there is some constant $k$ where for all $\alpha$
$$\frac k{26}=\pderiv{L}{\pi_\alpha}=\pderiv{\ell(\pi_\alpha, q_{1, \alpha})}{\pi_\alpha}+\pderiv{\ell(\pi_\alpha, q_{2, \alpha})}{\pi_\alpha}.$$
Taking $\ell=f$ and calculating shows
$$k=26\pderiv{L}{\pi_\alpha}=\ln\frac{\pi_\alpha}{1-\pi_\alpha}-\ln\frac{q_{1, \alpha}}{1-q_{1, \alpha}}+\ln\frac{\pi_\alpha}{1-q_\alpha}-\ln\frac{q_{2, \alpha}}{1-q_{2, \alpha}}$$
so 
$$\ln\frac{\pi_\alpha}{1-\pi_\alpha}-\frac12\left(\ln\frac{q_{1, \alpha}}{1-q_{1, \alpha}}+\ln\frac{q_{2, \alpha}}{1-q_{2, \alpha}}\right)$$
does not depend on $\alpha$. Computing $\bec \pi$ analytically from here is difficult, so some numerical algorithm should be used.
\subsection{\texorpdfstring{Method 2: Using an Asymmetric Loss Function}{Method 2: Using an Asymmetric Loss Function}}
Here we consider the loss function determined by
$$\mathcal D_i=\tilde D_{i, B_i} = \frac{1}{\#M_i}\sum_{S \in O_i} \sum_{\boldsymbol b \in S}\tilde{\ell}(P(\bec b), Q_i(\bec b))$$
Here $O_i\subset 2^{B_i}$ is the set of subsets $S$ of $B_i$ whose sum is $\bec 1$. In this case, $O_i$ is the singleton set $\set{B_i}$ as the heuristic's predictions are the maximal spanning set and their sum is $1$, so we sum over the same events as previously. However, instead of using the function $\ell$, we use $\tilde \ell$ and ignore the terms related to the surprise of a letter not being chosen. As $M_1=M_2=26$ is the number of terms in the summand, the loss function is
$$L(V\bec \pi, (\bec q_1,\bec q_2))=\frac{1}{26}\sum_{\alpha\in \set{\text{A}, \text{B}, \dots}} \tilde \ell(\pi_\alpha, q_{1, \alpha})+\tilde \ell(\pi_\alpha, q_{2, \alpha}).$$
As we are using $\ell=f$, so $\tilde \ell(p, q)=p\ln \frac pq$, this becomes
$$=\frac{1}{26}\sum_{\alpha\in \set{\text{A}, \text{B}, \dots}} \pi_\alpha \ln \frac{\pi_\alpha}{q_{1, \alpha}}+\pi_\alpha\ln \frac{\pi_\alpha}{q_{2, \alpha}}.$$
And again using Lagrange multipliers to take the derivative, the quantity
$$26\pderiv{L}{\pi_\alpha} - 2=\ln \frac{\pi_\alpha}{q_{1, \alpha}} + \ln \frac{\pi_\alpha}{q_{2, \alpha}}$$
does not depend on $\alpha$. Therefore, for some $k$ that does not depend on $\alpha$,
$$\pi_\alpha = k(q_{1, \alpha}q_{2, \alpha})^\frac{1}{2}.$$
Using the fact that $\sum_\alpha \pi_\alpha = 1$,
$$\pi_\alpha=(q_{1, \alpha}q_{2, \alpha})^\frac{1}{2}\left(\sum_\beta (q_{1, \beta}q_{2, \beta})^\frac12\right)^{-1}.$$
\subsection{Comparison of Methods}
Using a list of 10,000 common English words \cite{price_wordlist}, there are 7 letters that appear both after EM and before IL, not necessarily in the same word (A, B, E, M, O, P, S). Applying the methods described above yields Table \ref{fig:extended_ex}:
\begin{table}[h]
\begin{center}
         \begin{tabular}{|m{3em}|m{2em}|m{2em}| m{9em} |m{9em}| m{5em} |}
         \hline
         Letter &$\bec q_1$ & $\bec q_2$ & \multicolumn{2}{c|}{$\ell = f$} & $\ell = f^o$  \\\cline{4-5} 
        & & & $\bec p^*$ with $\mathcal D_i=D_{i, B_i}$ & $\bec p^*$ with $\mathcal D_i = \tilde{{\scriptstyle D}}_{i, B_i}$ & $\bec p^*$
         \\\hline $\begin{matrix}
             A \\ B \\ E \\ M \\ O \\ P \\ S
         \end{matrix}$ &
        $\begin{matrix}
    0.16\\0.08\\ 0.39\\ 0.01\\ 0.15\\  0.17\\  0.04
\end{matrix}$ & 
$\begin{matrix}
     0.32\\  0.27\\ 0.02\\  0.22\\ 0.03\\ 0.07\\ 0.07
\end{matrix}$ &$\begin{matrix}
            0.29 \\ 0.20 \\ 0.14 \\ 0.07 \\ 0.09 \\ 0.14 \\ 0.07
        \end{matrix}$
        &$\begin{matrix}
            0.31 \\ 0.20 \\ 0.13 \\ 0.05 \\ 0.09\\0.15\\0.07
        \end{matrix}$
        & $\begin{matrix}
            0.24\\0.18\\0.20\\0.12\\0.09\\0.12\\0.05
        \end{matrix}$
         \\\hline
    \end{tabular}
\end{center}
\caption{The resulting $\bec p^*$s when reconciling the heuristic $\bec q_1$ based on the 3-gram EM* and the heuristic $\bec q_2$ based on the 3-gram *IL with various methods.}
\label{fig:extended_ex}
\end{table}
\noindent \\\\Firstly, in this setup, when using $f^o$, it does not matter whether one includes the $\tilde f^o(1-p, 1-q)$ terms. In either case, the method amounts to setting $p_\alpha = \frac{q_{1, \alpha}+q_{2, \alpha}}{2}$, which is coherent as the set of coherent beliefs is convex. This will not always be the case, and was so here because of the simplicity of the events included in the sum. Looking at the results of using $f$, there is little difference, although low probabilities ($q_{1, m}=0.01$) seem to be given more weight when excluding $\tilde f(1-p, 1-q)$ terms. As usual, $f$ gives more extreme estimates than $f^o$ with either method.
\\\\
Here, both methods with either loss function were correct in assigning the letter A the highest probability in EMAIL. However, this is not usually the case. Over all the 5 letter words in \cite{price_wordlist}, both methods using $f$ gave the highest probability to the actual letter about $34\%$ of the time and both methods using $f^o$ were correct about $33\%$ of the time. The code to reproduce this experiment can be found at \url{https://github.com/scim142/quantifying_coherence}.
\end{document}